\documentclass[10pt]{article}
\usepackage{amsfonts,amssymb,pdfpages, xcolor}

 \usepackage{hyperref}
 
 \usepackage{amsmath}
\usepackage{amssymb}
\usepackage{indentfirst}
\usepackage{geometry}
\usepackage{amsfonts}
\usepackage{hyperref}

\newtheorem{theorem}{Theorem}[section] 

\newtheorem{lemma}[theorem]{Lemma}

\newtheorem{corollary}[theorem]{Corollary}

\newtheorem{definition}[theorem]{Definition}
\newtheorem{example}[theorem]{Example}

\newcommand{\qed}{\enspace\vrule  height6pt  width4pt  depth2pt}
\newenvironment{proof}{\par\noindent{\bf Proof.}}{$\qed$\par\bigskip}

\begin{document}

\title{Generalized Differential  Geometry} 
  
\author{   Juriaans, S. O.   \and  Queiroz, P.C.\footnote{IME-USP and UFMA, ostanley@usp.br, paulo.queiroz@ufma.br,    S\~{a}o Paulo-Brazil.}}

\date{}

\maketitle

\begin{abstract}Generalized Functions play a central role in the understanding of differential equations containing singularities and nonlinearities.  Introducing infinitesimals and infinities to deal with these obstructions leads  to controversies concerning  the existence, rigor and the amount of non-standard analysis needed to understand these theories.    Milieus constructed over the generalized reals sidestep them  all. A Riemannian manifold $M$ embeds   discretely into  a generalized manifold $M^*$ on which singularities vanish and products  of  nonlinearities make  sense.    Linking this to  an already existing global theory provides   an algebra  embedding $\kappa  :\hat{{\cal{G}}}(M)\longrightarrow {\cal{C}}^{\infty}(M^*,\widetilde{\mathbb{R}}_f)$.  Generalized Space-Time is constructed and its possible effects on Classical Space-Time are examined.  \end{abstract}

\section{Introduction} 

One of the main problems of mathematics is to construct environments in which equations have solutions.  They  need to  be constructed together with a set of mathematical tools that facilitates the constructions of  solutions and guarantee a way of obtaining  classical solutions if they do exists.  For differential equations involving singularities and nonlinearities  the construction of such milieus involves   the introduction  of infinitesimals and infinities and hence the endeavor   is not only  challenging but also leads to misconceptions about the amount of non-standard analysis one must master.

This does not need to be the case at all!    For example,  in \cite{Gio1, Gio2, Gio4}   such an endeavor is  undertaken. An environment,  whose basic  underlying structure is  $^{\bullet}\mathbb{R}$, the Fermat reals,  a totally ordered topological ring containing  infinitesimals,  is constructed in which nearly all classical features of Newtonian Calculus hold. Basically,  $^{\bullet}\mathbb{R}$ is the union of the halos of  elements  $^{\circ}x\in \mathbb{R}$  with  $halo(^{\circ}  x)=\{ ^{\circ} x+dt_a :\ a\in \mathbb{R}, a\geq 1\}$ and each $dt_a$ is a nilpotent element of order $min\{p\in \mathbb{N}\ :\ \frac{a}{p}<1\}$. In particular, $Inv(^{\bullet}\mathbb{R})$,  the group of invertible elements of $^{\bullet}\mathbb{R}$,   consists  of  the non-infinitesimals and hence,   $Inv(^{\bullet}\mathbb{R})$ is a group which is open but not dense in $^{\bullet}\mathbb{R}$.   Moreover, $\mathbb{R}$ is discrete in $^{\bullet}\mathbb{R}$ and,   because of its structure, $^{\bullet}\mathbb{R}$  contains no infinities. This should pose  some limitation on the Calculus over $^{\bullet}\mathbb{R}$   when dealing with problems involving certain singularities and nonlinearities both being   very common in Physics, Fluid Dynamics  and other areas. In such areas, infinitesimals and infinities have to coexist.    Hence  any milieu  aiming to deal with physical reality and  in which infinitesimals and infinities do coexist,  some of the latter and the former must be invertible elements. The  following   questions then arises: Can  such an environment  be constructed?    And  if so,   what is the amount of non-standard analysis  needed to master the basics of the construction?  If it does exists, how does the Calculus constructed in this environment  compare to the already familiar and so successful Newtonian Calculus and Schwartz's premier work on generalized functions? A reasonable way or tools to obtain classical solutions from the generalized solutions  must also be given.

In this paper we present     environments $\widetilde{\mathbb{R}}^n$ whose underlying structure is an ultrametric non-Archimedean   partially ordered  ring $\widetilde{\mathbb{R}}$ containing $\mathbb{R}$. Hence, our proposal is in the realm of non-Archimedean function theory born in 1961 in a Harvard seminar of J. Tate. It  is a specialized and advanced field of mathematics, requiring a solid foundation in mathematical analysis and  familiarity with nonstandard analysis. Tate's  remarkable discoveries led to a theory  which made possible analytic continuation over totally disconnected ground fields. The original context of the construction of  $\widetilde{\mathbb{R}}$  can be found in \cite{aragona1991intrinsic, Colom3,  Ober}, but these original sources do not contain  any of  the algebraic and topological ingredients which now do exists. 

To understand the construction of the non-Archimedean milieus  that   we propose, only working  knowledge in  Algebra,  Calculus and Distribution Theory  is needed. In spite of the nonstandard flavor, absolutely no familiarity with nonstandard analysis is required.  The ring  $\widetilde{\mathbb{R}}$, coined the {\it generalized reals},   has very nice properties some of which we enumerate:  it is a partially ordered ultrametric ring,   has   no nilpotent elements and its group of invertible elements,   $Inv(\widetilde{\mathbb{R}})$,  is open and dense. Given  $x\in \widetilde{\mathbb{R}}-Inv(\widetilde{\mathbb{R}})$,   there exist   $f, e\in {\cal{B}}(\widetilde{\mathbb{R}})$, the Boolean algebra of idempotents  of $\widetilde{\mathbb{R}}$, such that $e\cdot x=0$ and $f\cdot x\in Inv(f\cdot\widetilde{\mathbb{R}})$. Moreover,  $\mathbb{R}$  embeds  into $\overline{B}_1(0)\subset \widetilde{\mathbb{R}}$   as a grid of equidistant  points.  The  unit ball $B_1(0)$ consists of infinitesimals, as expected,  and $\widetilde{\mathbb{R}}$ contains an  isomorphic copy of the group $(\mathbb{R},+)$ which we denote by $\{\alpha_r\ : \alpha_r\cdot\alpha_s=\alpha_{r+s}, \ r\in \mathbb{R}\}$  all of which, except for  $1=\alpha_0$, are either infinities or infinitesimals. In   a subset of $\widetilde{\mathbb{R}}$,  Dirac's infinity, the delta Dirac function, becomes a differentiable function such that  $\delta(0)=\alpha_{-1}$, an infinity, and the zero distribution   $f(x)=x\delta$  becomes  a non-zero   differentiable function  with $f(0)=0$.   Also, differently than in \cite{Gio4}, the existence of primitives follows easily. So the milieus  we referred to  do exist and the Calculus developed  extends  in a very natural way classical  calculus and distribution theory. In this framework    theories of generalized functions  are brought back  to   Newton's  most fundamental insight:   the notion  of variation;  the single most important idea in Science and the essence of Nature.    One of the mathematical tools  to obtain a classical solution from a solution obtained in these generalized environments is the tool of support which can be defined in any such environment. It naturally extends the notion of association  or shadow from distribution theory.

The layout of the paper is thought out  to first construct these environments in a simple and clear way so that the reader can get acquainted with  its  construction. For this,  we first  construct what are known as the simple environments.  After this,  we switch gear and look at what are known as the full environments with underlying basic  structure  $\widetilde{\mathbb{R}}_f$. The latter  are needed  if one looks for   a canonical embedding of  ${\cal{D}}^{\prime}(\Omega)$. It is in them  that we prove a fixed point theorem and construct a generalized differential manifold from a classical Riemannian manifold. In particular, we obtain  Generalized Space-Time, a generalized version of Classical Space-Time in which infinities and infinitesimals naturally coexists and which contains Classical Space-Time as a discrete and bounded grid of equidistant points.

In the next section,  we explain the construction of the generalized reals  $\widetilde{\mathbb{R}}$  and give its main  properties.  In the third section,  we introduce functions defined on subsets of $\widetilde{\mathbb{R}}^n$  and show how to embed classical functions in a   natural way into  these new algebras. In the fourth section,  we   define  the notion of differentiability  being  a natural extension of Newton's Calculus.  In the fifth  section   new ideas and machinery   are proposed, constructed and used. As an application we prove a Fixed Point Theorem and the {\it Down Sequencing Argument}.

\begin{theorem}{\bf{[Fixed Point Theorem]}}

Let $\Omega\subset\mathbb R^N$, $A=[(A_\varphi)_\varphi]\subset B_{r_0}(0)\cap\mathcal G_f(\Omega)$, $r_0<2$, be an internal set, and $T:A\rightarrow A$ be given with representative $(T_\varphi:A_\varphi\rightarrow A_\varphi)_{\varphi\in {\cal{A}}_0}$. If there exists $k=[(k_\varphi)_\varphi]\in\widetilde{\mathbb N}$ such that   $T^k=(T_\varphi^{k_\varphi})$ is a $\lambda$-contraction, then $T$ is well defined,     continuous and has a unique fixed point in $  A$.

\end{theorem}

\begin{theorem}{\bf{[Down Sequencing Argument]}}
Let $f\in\mathcal G_f(\Omega)$,    $f\in W^0_{l,4^kr}[0]$ with $r>0$.  Then $f\in W^{k}_{l,r}[0]$, i.e.,  $W^0_{l,4^kr}[0]\subset W^{k}_{l,r}[0]$.
\end{theorem}

 \noindent In section six,  we continue in the full version of these milieus and show how to construct a generalized manifold $M^*$  from a classical Riemannian  manifold $M$ and prove an embedding result extending  results of \cite{aragona2005discontinuous} by linking them  to  results obtained in \cite{gkos4, invariant}. The construction of $M^*$ involves several   developments in the field.  Consequently, we extend some  former results,  obtained for open subsets of $\mathbb{R}^n$,  to abstract manifolds.

\begin{theorem} {\bf{[Embedding Theorem]}} Let $M$ be an $n-$dimensional orientable  Riemannian manifold. There exists an $n-$dimensional ${\cal{G}}_f-$manifold $M^*$  and an algebra monomorphism 

$$\kappa  :\hat{{\cal{G}}}(M)\longrightarrow {\cal{C}}^{\infty}(M^*,\widetilde{\mathbb{R}}_f)$$

\noindent which commutes with derivation. Moreover, $ssupp(M^*)=\overline{M}$, the topological closure of $M$,   and  equations defined on $M$, whose data have singularities or nonlinearities,    naturally extend to equations on $M^*$   where these data become $\ {\cal{C}}^{\infty}-$functions.

\end{theorem}

Finally, we   extend   Generalized Differential  Calculus to modules over the generalized reals  and the generalized complex numbers  in the full setting. The effects of Generalized Space-Time on Classical Space-Time is considered as is the effect of the existence of generalized solutions of differential equations on numerical solutions of these equations.  Absolutely no claim is made that Generalized Space-Time  corresponds  to physical reality.   References on which this introduction  is based are \cite{meril,  aragona1991intrinsic,  aragona2005discontinuous, aragona2009natural, hebe, bosh, Colom3, juriaans2022fixed,  inter, ku,  Ober, Schwartzlivre, wolfram}. Notation used  is standard.

\section{The underlying structure $ \widetilde{\mathbb{R}}$}

In this section we  construct    the ring $\widetilde{\mathbb{R}}$ which  is the basic underlying structure of the environments  to be created. Its construction is without the use of   nonstandard methods and  basically mimics the sequential construction of the reals starting from the rationals, the reason to coin it the {\it generalized reals}. It is  constructed  in such a way that whatever norm  used  in $\mathbb{R}^n$   results in the same  norm in $\widetilde{\mathbb{R}}^n$.

Let us  first explain the construction in terms of the sequential  construction of $\mathbb{R}$ starting from $\mathbb{Q}$,  but enlarging a bit the initial ring and  factoring out by an ideal so that  infinities and infinitesimals are introduced in the resulting quotient ring.  Our initial ring consists of   rational sequences of polynomial growth. There is a reasonable explanation for the why of this model:  this has to do with  the notion of computability which in turn relates to the notion of  time, thus  making   the model more in sync with   physical reality since infinities whose growths are bigger then polynomial growth should not exist in physical reality, but, on the other hand,  infinitesimals can be used to perform calculations of exponential growth.    Let's look at the math.  We consider rational sequences and the partial order $(x_n)\leq (y_n)$ if there exists $n_0\in \mathbb{N}$ such that $x_n\leq y_n,\ \forall\  n\geq n_0$. Consider the sequence $ \tau =(\frac{1}{n})$,  $\tau_k =\tau^k=(\frac{1}{n^k}), k\in \mathbb{Z}$ and the algebra ${\cal{A}}=\{ (x_n)_{n\in \mathbb{N}} \ :\ x_n \in \mathbb{Q} , \exists\  N\in \mathbb{N}, \ \mbox{such that}\ (|x_n|)\leq \tau_{-N}\}$. Then ${\cal{J}}=\{(x_n)\in {\cal{A}}\ :\  (|x_n|)\leq \tau_N,\ \forall\ N\in \mathbb{N} \}$ is a radical  ideal of ${\cal{A}}$. The image of the Cauchy sequences in the  quotient algebra $\widetilde{\cal{A}}={\cal{A}}/{\cal{J}}$ is  a copy of $\mathbb{R}$ and the image of the bounded sequences (compactly supported sequences)  is a subring containing the copy of $\mathbb{R}$. The nontrivial elements of the Boolean algebra of idempotents,  ${\cal{B}}(\widetilde{\cal{A}}) $,   are     characteristic functions of special subsets of $\mathbb{N}$ and the image of the powers of $\tau$ are infinities and infinitesimals. Infinitesimals have representatives  which converge to zero.   Given an  element $x$ with a bounded representative,  there exist  idempotent   $e\in {\cal{B}}(\widetilde{\cal{A}})$ and $x_e\in \mathbb{R}$ such that  $ex -ex_e$ is an infinitesimal. The element $ex$ is nothing more than a subsequence of $x$, i.e., multiplying an element  by an idempotent corresponds to taking a subsequence of that element and $x_e$ is an accumulation point of $x$.  For any element $x\in \widetilde{\cal{A}}$, define  $supp(x)=\{ x_e\in \mathbb{R}\ :\ ex -ex_e\ \mbox{is an infinitesimal}, e\in {\cal{B}}(\widetilde{\cal{A}})\}$. If the support consists of only one element $x_{e_0}$ we say that $x_{e_0}$ is the shadow of $x$, i.e., the sequence converges.  It is not hard to see that in this algebra elements are either units or zero divisors and that zero divisors have  idempotents in their  annihilators. The prime and  maximal  ideals are related to ultrafilters  in   ${\cal{B}}(\widetilde{\cal{A}})$. So it is the parameter space $\mathbb{N}$  that gives rise to idempotents, filters, the zero set of elements and how elements are interlinked. Compactification   is not needed.  The reader should have  this example in mind when reading the rest of this section.   Let's  now continue with the construction of $\widetilde{\mathbb{R}}$ which mimics what  we just did.  

 Let $I =]0,1]$ and ${\cal{E}}(\mathbb{R})={\cal{F}}(I,\mathbb{R})$ whose elements are called nets and denoted either by $(x_{\varepsilon})$ or  $(x_{\varepsilon})_{\varepsilon}$.  The algebra ${\cal{E}}(\mathbb{R})$ has  a natural partial order induced by the comparison of values of functions, more precisely: we say that $(x_{\varepsilon})\leq (y_{\varepsilon}) $ if there exists $\eta\in I$ such that $x_{\varepsilon}\leq y_{\varepsilon}, \forall \varepsilon\in I_{\eta}:=]0,\eta]$.  Given  a net $x=(x_{\varepsilon})$, we denote by $|x|$ the net $(|x_{\varepsilon}|)$.    We denote by $\alpha =(\varepsilon)_{\varepsilon}$,   coining it the {\it natural gauge} or {\it standard gauge}, and by $\alpha_r :=\alpha^r$, $r\in \mathbb{R}$. Since $\alpha_r\cdot \alpha_s=\alpha_{r+s}$,  it follows that the map $(\mathbb{R},+)\ni  r\longrightarrow \alpha_r \in {\cal{E}}(\mathbb{R})$ is a group monomorphism. This fundamental gauge will be key in all the constructions. All growth will be measured by  this gauge,  resulting in the rejection of certain infinities and the avoidance of some infinitesimals.  At the end of this section we shall see that it is related to the way we measure in Classical Analysis. At first, one could think of using other gauges, but, just like in the case of the Principle of Computational Equivalence, one should look for the gauge that permits interpretation of physical reality without overcomplicating.

An element   $x \in {\cal{E}}(\mathbb{R})$ is moderate if $|x| <\alpha^r$, for some $r\in \mathbb{R}$.    Denote the set of moderate nets  by ${\cal{E}}_M$ and  by ${\cal{I}}=\{ x : x\in {\cal{E}}_M, \ |x|<\alpha^n, \forall\  n \in \mathbb{N}\}$. For  $x\in {\cal{E}}_M$,  denote by $V(x)=Sup\{r\in \mathbb{R} : |x|<\alpha^r\}$ and set $\|x\| =e^{-V(x)}$.  It is easily seen that  $\cal{I}$ is a radical  ideal of the ring  ${\cal{E}}_M$ and setting  $\widetilde{\mathbb{R}} : =\frac{{\cal{E}}_M}{\cal{I}}$, we have that $(\widetilde{\mathbb{R}} , \| \ \|)$ is a partially ordered  ultra-metric     ring containing a copy of $\mathbb{R}$. It is common to denote elements of ${\cal{E}}_M$ by $\hat{x}$ and its class in $\widetilde{\mathbb{R}} $ by $x$ and call $\hat{x}$ a representative of $x$.  We shall call $\widetilde{\mathbb{R}} $ the ring of {\it generalized reals} (or {\it Colombeau reals}) and, to understand it better,  one should think in  terms of germs at $0$, although $0\notin I$ (see \cite{Colom3, Ober}).    The set of invertible elements of $\widetilde{\mathbb{R}} $ is denoted by $Inv(\widetilde{\mathbb{R}})$ and its boolean algebra of idempotents   is denoted by  ${\cal{B}}(\widetilde{\mathbb{R}})=\{e\in\widetilde{\mathbb{R}}\ :\  e^2=e\}$. The following theorem summarizes all basic properties one needs to know about   $\widetilde{\mathbb{R}}$. Just as one rarely uses the way $\mathbb{R}$ is constructed from $\mathbb{Q}$, in general, one does not need to remember the construction of $\widetilde{\mathbb{R}}$. Most proofs can be done intrinsically without appealing to representatives. 

\begin{theorem} {\bf{[The Extended Fundamental Theorem of $\ \widetilde{\mathbb{R}}$]}} $ \ $
\begin{enumerate}
\item $\widetilde{\mathbb{R}}$ is a partially ordered non-Archimedean  ultrametric  algebra such that $\|x\cdot y\|\leq \|x\|\cdot \|y\|$. In particular, $\|x+y\|\leq max\{\|x\|,\|y\|\}$,  $\|rx\|=\|x\|$, $r\in\mathbb{R}^*$  and $\widetilde{\mathbb{R}}$ is a totally disconnected topological ring.
\item   $Inv(\widetilde{\mathbb{R}})$ is open and dense in $\widetilde{\mathbb{R}}$. In particular, maximal ideals are closed and rare.
\item  $x\in Inv(\widetilde{\mathbb{R}})$ if and only if $|x|\geq \alpha_r$, for some $r\in \mathbb{R}$.
\item  ${\cal{B}}(\widetilde{\mathbb{R}})=\{\chi_S \ :\  S\subset {\cal{S}}\}$, where $\chi_S$ is the characteristic function of $S$ and $S\in {\cal{S}}$ if and only if $0$ is in the topological closure, in $\mathbb{R}$,  of both  $S$ and $S^c=I-S$.  
\item $x\in \widetilde{\mathbb{R}}-Inv(\widetilde{\mathbb{R}})$  if and only if there exists an idempotent $e\in {\cal{B}}(\widetilde{\mathbb{R}})-\{0,1\}$ such that $e\cdot x=0$.
\item The Jacobson Radical of $\ \widetilde{\mathbb{R}}$ is $\{0\}$. In particular, $\widetilde{\mathbb{R}}$ has no nontrivial nilpotent elements and is embeddable into  a product of integral domains. 
\item The unit ball $B_1(0)=\{x\in \widetilde{\mathbb{R}}\ :\ \|x\|<1\}$ consists of infinitesimals, i.e., if $x\in B_1(0)$ then  $|x| < \frac{1}{n}, \ \forall \ n\in \mathbb{N}$. In particular, $\alpha_r\in B_1(0)$, $\forall\  r>0$.
\item For each $r<0$, $\alpha_r$ is an infinity, i.e., $\alpha_r > n,\ \forall n\in \mathbb{N}$.
\item If $x\in \widetilde{\mathbb{R}}-\{0\}$ then there  exists $e\in  {\cal{B}}(\widetilde{\mathbb{R}})$ such that $e\cdot x\in Inv(e\cdot  {\widetilde{\mathbb{R}}})$.
\item The  {\it Biagioni-Oberguggenberger topology} of $\ \widetilde{\mathbb{R}}$,  the  {\it sharp topology},  is generated by the  balls with generalized radii $V_r(0)=\{x\in \widetilde{\mathbb{R}}\ :\  |x| <  \alpha_r\}$.  

\end{enumerate}
\end{theorem}

The reals, $\ \mathbb{R}$,  are  embedded  as a grid of equidistant  points in the generalized reals  $\widetilde{\mathbb{R}}$, the common distance being equal to one (one can adjust the ultra-metric so that this common distance is at the scale of uncertainty in physical reality). The ideals of  $\widetilde{\mathbb{R}}$  are convex (an ideal $J$ is convex if $x\in J$,  $y\in\widetilde{\mathbb{R}}$ and $|y|\leq |x|$ implies that $y\in J$), its Krull dimension is infinite and it has a minimal prime  which is also a maximal ideal.  Ultrafilters of $\ {\cal{S}}$ partially parametrize prime and maximal ideals of $\ \widetilde{\mathbb{R}}$.        The halo of $x\in \widetilde{\mathbb{R}} $ is defined as $halo(x)=B_1(x)=x+B_1(0)$. We say that $x$   and $y$ are associated (denoted by $x\approx y$) if $x-y$ is an infinitesimal  or equivalently $\lim\limits_{\varepsilon\longrightarrow 0}(\hat{x}-\hat{y})(\varepsilon)=0$. In particular, if $y\in halo(x)$ then $y\approx x$.

The environments we propose are $\widetilde{\mathbb{R}}^n$ with $n\in \mathbb{N}$. Note that, in some sense, points  in these environments are not static  since one must think of them in terms of  germs  when $\varepsilon \downarrow 0$.  Since norms in $\mathbb{R}^n$ are all equivalent, they all induce the same topology in $\widetilde{\mathbb{R}}^n$.     We say that a point $p=(p_1,\cdots, p_n)\in \widetilde{\mathbb{R}}^n$ is {\it compactly supported} if there exists $L\in \mathbb{R}$ such that $\|p\|_1\leq L$, in $\widetilde{\mathbb{R}}$. We also call them {\it finite points}. The set of finite points is denoted by $\widetilde{\mathbb{R}}^n_c$ and is an $\widetilde{\mathbb{R}}_c-$ module. Thus, equipping $\widetilde{\mathbb{R}}^n$ with the product topology,  it follows that $\widetilde{\mathbb{R}}^n_c\subset \overline{B}_1(0)=\{x\in \widetilde{\mathbb{R}}^n\ :\  \|x\|\leq 1\}$. 

 Given $\Omega \subset \mathbb{R}^n$, define  $\widetilde{\Omega}:=\{p\in \widetilde{\mathbb{R}}^n\ :\ \exists \ \eta\in I,  \hat{p}=(p_{\varepsilon}),  p_{\varepsilon}\in\Omega, \ \mbox{for}\ \varepsilon\in I_{\eta}\}$        and  $\widetilde{\Omega}_c:=\{p\in \widetilde{\mathbb{R}}^n_c\ :\ \exists \ \eta\in I,  \hat{p}=(p_{\varepsilon}),  p_{\varepsilon}\in\Omega, \ \mbox{for}\ \varepsilon\in I_{\eta}\}=\widetilde{\Omega}\cap \widetilde{\mathbb{R}}^n_c$. Given a point $p\in \widetilde{\mathbb{R}}^n$ we define its {\it support} as $ supp(p)=\{q\in \mathbb{R}^n\ :\ \exists \ e\in  {\cal{B}}(\widetilde{\mathbb{R}}), \ \mbox{such that}\ e\cdot p\approx e\cdot q\}$.  For a subset $X\subset   \widetilde{\mathbb{R}}^n_c$ we define its support as $supp(X)=\bigcup\limits_{p\in X}supp(p)$.        If $\| p\|<1$ then $supp(p)=\{\vec{0}\}$, $supp(\alpha_{-1})=\emptyset$, $supp([sin(\alpha_{-1})])=[-1,1]$ and if $p\in  \widetilde{\mathbb{R}}^n_c$,  then $supp(p)\neq \emptyset$.  It is clear  that $supp(\widetilde{\Omega}_c)=\overline{\Omega}$, the topological closure of $\Omega$ in $\mathbb{R}^n$.     These   notions are   related to  {\it  interleaving}, {\it membranes} and {\it internal sets}  for which we refer the reader to \cite{juriaans2022fixed, oberguggenberger2008internal}.  The notion of support can be extended to operators acting on these environments and be used to obtain classical solutions, if they exist, for classical operators in the support of an operator in these environments: If $L$ is a generalized operator, $L_0\in supp(L)$, $F$ a generalized solution of the equation $L(\tau)=0$, then $L_0(f_0)=0$, for $f_0\in supp(F)$. In fact, there exists an idempotent $e\in {\cal{B}}(\widetilde{\mathbb{R}})$ such that $f_0=e\cdot F$.  So this serves as a mathematical tool to obtain classical solutions. This shows the advantage of ${\cal{B}}(\widetilde{\mathbb{R}})$ not to be trivial, which is not  the case when the underlying structure is a fields or the Fermat reals.
 
 We finish this section comparing  the classical topology in Analysis  with the sharp topology. This is best done looking at the definition of continuity.  Let $f:\mathbb{R}\longrightarrow\mathbb{R}$ be  continuous at $x_0$.  Classically  this means that for each $\varepsilon$ there exists $\delta_{\varepsilon}$  such that whenever $|x-x_0|\leq \delta_{\varepsilon}$ we have that $|f(x)-f(x_0)|\leq \varepsilon$  ($|x-x_0|\leq \delta_{\varepsilon} \Longrightarrow |f(x)-f(x_0)|\leq \varepsilon$).  Translated to the sharp topology this can be written as: $f$ is continuous at $x_0$  if given $\alpha$, there exists $\delta=[(\delta_{\varepsilon})]$ such that    $x\in V_{\delta}(x_0)\Longrightarrow f(x)\in V_1(f(x_0))$, where $V_{\delta}(x_0)=\{x\in \widetilde{\mathbb{R}}\ :\ |x-x_0|<\delta\}$.   This proves that  classically we stop measuring at scale $\alpha$, this being the reason why we coined  it our natural gauge.   References on which this section is based are \cite{aragona1991intrinsic, aragona2009natural, Colom3, juriaans2022fixed,  ku, Nedel, scarpalezos2000colombeau, wolfram}.

\section{Functions on $\widetilde{\Omega}_c$}

We  now define functions on subsets of $\widetilde{\mathbb{R}}_c^n$ in such a way that it naturally extends our classical definitions, but at the same time permits that we can see both classical functions and distributions as functions in these new environments, i.e., we are looking for a domain for these objects. More important, composition of   functions, if the classical conditions are satisfied,  must be  defined.   The way to achieve this is to use our natural gauge.

Let $\Omega \subset \mathbb{R}^n$ be an open subset and let  $(\Omega_m)$ be an exhaustion of compact subsets  of $\Omega$. Then   $(\widetilde{\Omega}_{mc})$, where $\widetilde{\Omega}_{mc}=(\widetilde{\Omega_m})_c$,  is an exhaustion of $\widetilde{\Omega}_c$  by   these  subsets, called {\it principle  membranes},  of  $\widetilde{\Omega}_c$.

Given     a net  $\hat{f}=(\hat{f}_{\varepsilon})\in {\cal{C}}^{\infty}(\Omega,\mathbb{R})$ and $  \hat{p}=(p_{\varepsilon})$,  consider the net $\hat{f}(\hat{p})=(\hat{f}_{\varepsilon}(p_{\varepsilon}))\in{\cal{E}}(\mathbb{R})$.  For  $\hat{q}=(q_{\varepsilon})$, we have that $|\hat{f}(\hat{p}) -\hat{f}(\hat{q})|\leq  \|\nabla\hat{f}(\hat{p}_1)\|_2\cdot \|\hat{p}-\hat{q}\|_2$, which will  make $\hat{f}$ a well defined function on $\widetilde{\Omega}_c$ if  imposed   a growth condition on both   $\hat{f}$ and  $\nabla\hat{f}$. 

  We say that $\hat{f}$ is $\alpha-${\it bounded} (or of {\it moderate growth})  if  for  each $m\in \mathbb{N}$ there exists $r=r(m)\in \mathbb{R}$ such $\{ \hat{f}(\hat{p}), \partial_{x_i}\hat{f}(\hat{p})  \ :\ i=1,n,\  \hat{p}\in  \Omega_m\} \subset V_{r(m)}(0)$.    It follows that if    $\hat{f}$  is $\alpha-${\it bounded}  then it defines a function $\hat{f}:\widetilde{\Omega}_c\longrightarrow  \widetilde{\mathbb{R}}$, $\hat{f}(p)=[\hat{f}(\hat{p})]$ and $\hat{f}(\widetilde{\Omega}_{mc})\subset V_{r(m)}(0)$.  Let  ${\cal{E}}( \Omega,\mathbb{R})=\{\hat{f}\in  {\cal{C}}^{\infty}(\Omega,\mathbb{R})\ :\ \hat{f}\ \mbox{is}\ \alpha-\mbox{bounded}\}$, i.e., we look at the set of all  nets which define functions  on $\widetilde{\Omega}_c$ taking values in $\widetilde{\mathbb{R}}$.    Clearly,   $J=\{\hat{f}\in {\cal{E}}( \Omega,\mathbb{R}) \ : \ \hat{f}=0\  \mbox{on}\ \widetilde{\Omega}_c\}$ is an ideal of $ {\cal{E}}( \Omega,\mathbb{R})$ and ${\cal{E}}( \Omega,\mathbb{R})/J  $ embeds into ${\cal{F}}(\widetilde{\Omega}_c,\widetilde{\mathbb{R}})$.  With the notation of the previous section,  one can write $J= \{\hat{f}\in {\cal{E}}( \Omega,\mathbb{R}) \ : \ \hat{f}(\hat{p})\in{\cal{I}}, \  \mbox{for all}\ \hat{p}\in \widetilde{\Omega}_c\}$.      For example, take $n=1$,   let $\varphi\in {\cal{D}}(\mathbb{R})$ and consider the net $\hat{f}=(\varphi_{\varepsilon})$, where $\varphi_{\varepsilon}(x)=\frac{1}{\varepsilon}\varphi(\frac{x}{\varepsilon})$,  a contraction of $\varphi$. Then it is easily seen that $\hat{f}\in {\cal{E}}( \Omega,\mathbb{R})$ thus proving that  ${\cal{F}}(\widetilde{\Omega}_c,\widetilde{\mathbb{R}})$  is nonempty.  In fact, $r= -2.1$ serves  to prove that $\hat{f}$ is $\alpha-$bounded.

Define  $ {\cal{E}}_M( \Omega,\mathbb{R})=\{\hat{f}\in {\cal{E}}( \Omega,\mathbb{R})\ :\  \partial^{\beta}\hat{f}\in {\cal{E}}( \Omega,\mathbb{R})   ,\ \forall \ \beta \in \mathbb{N}^n\}$  and  ${\cal{N}}(\Omega)=\{\hat{f}\in J\ :\  \partial^{\beta}\hat{f}\in J,\ \forall \ \beta \in \mathbb{N}^n\}$. Then the former is a subalgebra  of $ {\cal{E}}( \Omega,\mathbb{R})$     containing the latter as a  radical ideal and we  have an embedding of  the quotient algebra $\kappa : {\cal{G}}(\Omega)=\frac{{\cal{E}}_M( \Omega,\mathbb{R})}{{\cal{N}}(\Omega)}\longrightarrow {\cal{F}}(\widetilde{\Omega}_c,\widetilde{\mathbb{R}})$.  Let $\tau\in {\cal{S}}(\mathbb{R})$ be an element of the Schwartz Space which is a nonzero constant in an interval containing $0$ and  denote by  $\rho$  its inverse Fourier transform. Then $(\rho_{\varepsilon})$ defines an element   $\kappa(\rho)\in {\cal{G}}(\mathbb{R})$ which can be  identified  with $\delta$, Dirac's delta function. So here $\delta=\kappa(\rho)=\alpha_{-1}\cdot \rho\circ\alpha_{-1}$ becomes  a function!  and $\delta(0) = \kappa(\rho)(0)=[(\frac{1}{\varepsilon}\rho(0))]= \rho(0)\cdot\alpha_{-1}$, an infinity. We have that $T=x\delta=0$ in ${\cal{D}}^{\prime}(\mathbb{R})$,  but $\kappa(x)\kappa(\delta)$ is a non-zero function. In fact, $\kappa(T)(x_0\alpha)=x_0\rho(x_0 )$.   The function $\rho$ is called a {\it  mollifier}  because of the following two moments properties which follow directly from  properties of the Fourier transform: $\int\limits_{\mathbb{R}}\rho dx =1$ and $\int\limits_{\mathbb{R}}x^k\rho dx =0$, for all $k\in \mathbb{N}, k\neq 0$. These properties are crucial to embed    ${\cal{D}}^{\prime}(\Omega)$ into ${\cal{G}}(\Omega)$ and thus into ${\cal{C}}^{\infty}(\widetilde{\Omega}_c,\widetilde{\mathbb{R}})$, i.e.,   distributions become infinitely differentiable functions. It is now clear that to compose two such functions $f$ and $g$, one must have that, just  like in the classical sense, $Im(g)\subset Domain(f)$. In particular, there exists a real number  $L>0$ such that $\|g(x)\|\leq L$ in $\widetilde{\mathbb{R}}$  for all $x\in Domain(g)$, i.e., $Domain(g)\subset \widetilde{\mathbb{R}}_c^n$.  We shall now  introduce our notion of derivation, thus going back to  Newton's and Nature's most basic idea of viewing physical  reality: the notion of variation.    References on which this section is based are \cite{aragona1991intrinsic, hebe, Colom3, ku, Nedel, Ober}.

\section{Differentiability in $\widetilde{\mathbb{R}}^n$} 

In this section we shall define differentiability in the   milieus  constructed.  Milieus  which somehow are the union of each single  computable path or event we can imagine.   The challenge is how to capture variation in such  environments  whose underlying structure  was turned into a  totally disconnected topological algebra with zero divisors.      It must be done intrinsically to stand on its own. At the same time, it must extends, in a natural way,  the classical notion of variation. J. Tate must have had the same challenge when he founded non-Archimedean function theory on totally disconnected topological fields.  How to achieve  such a thing in a totally disconnected ultrametric ring where spheres are open sets and  zero divisors do exists?   The solution is to use the natural gauge, just as it was used in the previous sections to define the concepts of  moderateness, domains  and functions in these milieus.      We basically  start to  measure where classical measurement ends: at scale $\alpha$!  Without further ado, here is our definition. 

Let $\Omega\subset \mathbb{R}^n$ be open.  We say that $f\in {\cal{F}}(\widetilde{\Omega}_c,\widetilde{\mathbb{R}})$ is differentiable at a point $p\in \widetilde{\Omega}_{c}$ if there exists a vector $v\in  \widetilde{\mathbb{R}}^n$ such that 

$$\lim\limits_{h\longrightarrow 0}\frac{f(p+h)-f(p) -\langle h \mid v \rangle }{\alpha_{-\ln(\|p-h\|)}} =0    $$

\noindent Note that

$$\Big\| \frac{f(p+h)-f(p) -\langle h \mid v \rangle }{\alpha_{-\ln(\|p-h\|)}} \Big\|=  \frac{\|f(p+h)-f(p) -\langle h\mid v \rangle \|}{\|p-h\|}$$
\vspace{.1cm}

 \noindent showing that our definition is very close to the Newtonian notion of differentiability.  The last equality holds because the $\alpha_r$'s are multiplicative elements of $\widetilde{\mathbb{R}}$ in the sense of \cite{bosh}. The vector $v$, if it exists,  is  unique and  differentiability at a point implies also  continuity at that point.  The resemblance to Newton's derivation  is even closer than meets the eye.

 \begin{example} Let  $\hat{f}_{\varepsilon}=f_0\in {\cal{C}}^{\infty}(\mathbb{R}), \hat{g}_{\varepsilon}=f_0^{\prime},\ \forall \varepsilon\in I$ and set $f=\kappa(\hat{f})$ and $g=\kappa(\hat{g})$.  Then $f^{\prime} =g$, i.e.,  $(\kappa(\hat{f}))^{\prime}=\kappa(\hat{f}^{\prime})$.
\end{example}

The example not only shows that our notion of Calculus extends naturally that of  classical Calculus but also gives   the existence of  primitives. This will become more   clearer  from our Embedding Theorem.   For the Calculus developed in the Ring of Fermat reals this is not as direct and obvious as in our case (see \cite{Gio4}). Partial derivatives are defined in an obvious way as is integration on $\Omega$. Integration on other sets,  such as membranes and internal sets,  is also defined in a natural and obvious way.  The following result  shows that our notion of variation and that of Newton are a perfect mirror match. Moreover, it shows the existence of an algebra of functions in which the linear space of Schwartz's distributions can be embedded as infinitely differentiable functions defined on  a principal membrane, an example of  an internal set,   both of which are  subsets of a Cartesian product  of the generalized reals.  In particular, the product of distributions does not only makes sense,  but is well defined in a classical sense as the product of  functions.

\begin{theorem}{\bf{[Embedding Theorem]}}
There exists  an algebra embedding\\   $\kappa : {\cal{G}}(\Omega) \longrightarrow {\cal{C}}^{\infty}(\widetilde{\Omega}_c,\widetilde{\mathbb{R}})$ which commutes with derivation, i.e., $\kappa\circ\partial^{\beta}=\partial^{\beta}\circ\kappa$, for all $\beta\in \mathbb{N}^n$. When restricted to the image of $ {\cal{D}}^{\prime}(\Omega)$ in ${\cal{G}}(\Omega)$, $\kappa$   is $\mathbb{R}-$linear and,  when restricted to the image of  ${\cal{C}}^{\infty} (\Omega)$,   it  is an algebra monomorphism. Moreover $supp(\widetilde{\Omega}_c)=\overline{\Omega}$.

\end{theorem}

As a consequence of the theorem, we have that  if $H$ is the Heaviside function, then  $\kappa(H)$ is a ${\cal{C}}^{\infty}-$function, with $H^{\prime}=\delta$ and  $\kappa(H)\notin {\cal{B}}(\widetilde{\mathbb{R}})$, i.e., $\kappa(H)^2\neq \kappa(H)$.  For $n=1$, the function $g(x)=\alpha_{\ln(\|x\|^{-2})}, g(0)=0$,   is differentiable, non-constant and $g^{\prime}(x)=0$ for all $x$. However,  if $f\in \kappa({\cal{G}}(\mathbb{R}))$ and $f^{\prime}\equiv 0$,  then $f$ is constant. Hence this fundamental classical rule is not violated if we restrict to function coming from ${\cal{G}}(\mathbb{R})$.   It is natural to call $g$  a quanta, since $F(x)=x\alpha_{\ln(\|x\|^{2})}, F(0)=\infty$ changes the sphere on which $x$ was. It acts like the inversion in the unit sphere:  $\|x\|\cdot \| F(x)\|=1$.  Other quanta  are the multiplication by a fixed $\alpha_r$.  

 If $T\in {\cal{D}}^{\prime}(\Omega)$ has compact support, then $\kappa(T)$ is constructed by convolution with the mollifier $\rho$ which immediately gives us the following remarkable result. In these milieus, the action of a distributions on ${\cal{D}}(\Omega)$ is   represented by a globally defined function! Thus revealing its true linearity.

\begin{lemma}
Let $T\in {\cal{D}}^{\prime}(\Omega)$  and $\varphi \in {\cal{D}} (\Omega)$. Then, in $\widetilde{\mathbb{R}}$,  we have that 

$$\langle T\mid \varphi\rangle =\int\limits_{\Omega}\kappa(\varphi)\kappa(T)d\Omega $$

\end{lemma}

 This lemma was proved for an open subset $\Omega\subset \mathbb{R}^n$ in   \cite{ku} and used in \cite{walter} to embed ${\cal{D}}^{\prime}(\Omega)$ into an {\it Aragona algebra}, a sub-algebra of  ${\cal{C}}^{\infty}(\widehat{\Omega},\mathbb{L})$, $\widehat{\Omega}\subset \mathbb{L}^n$, where $\mathbb{L}$ is a totally ordered  non-Archimedean ultra-metric  field of characteristic zero.    Differentiability readily extends to vector valued functions and thus we have this new Calculus available with  the same classical features and which, although $\mathbb{R}^n$ is discretely embedded into $\widetilde{\mathbb{R}}^n$,  extends the latter in a very natural way.    References on which this section is based are \cite{aragona2005discontinuous, aragonamembranas, AJ, walter,  ku}.

\section{ A Fixed Point Theorem} 

	 	
	In this section we shall be working in what are known as the full environments  which  permit   ${\cal{D}}^{\prime}(\Omega)$ to be embedded canonically into ${\cal{G}}(\Omega)$. Those   of the previous sections are  known as the simple milieus  and the embedding of  ${\cal{D}}^{\prime}(\Omega)$ into ${\cal{G}}(\Omega)$ depends on the mollifier $\rho$. This can be a feature since one may adapt $\rho$ to the specific problem one is interested in.  Although the construction in the full case  is more elaborate, everything is very similar to the simplified construction and the results on the basic underlying, in this case denoted by  $\widetilde{\mathbb{K}}_f$,  are identical to what we already saw. We also do have an embedding theorem and the lemma of the previous section also holds. The reader can find more in   \cite{ juriaans2022fixed, inter}  where results in this and the next   sections are proved in   $\widetilde{\mathbb{R}}$. The best way to master  the math in these milieus is to master their algebraic theory and understand well how the topology interlinks with the former.  Nearly everything else works  as  in the classical theory.   That is why  we start by reviewing the {\it sharp topology}    in these full milieus. 
	
	The sequence   $({\cal{A}}_q)_{q\in \mathbb{N}}$ stands for a decreasing chain of sets, with empty intersection,   consisting of functions  $\varphi$ (recall the definition  of its contraction $\varphi_{\varepsilon}$), of compact support,  with the same property at the origin as the function that defined the mollifier  $\rho$. The natural gauge in $\widetilde{\mathbb{K}}_f$ is   denoted by $\alpha^\bullet$. One of its representatives is   $\hat{\alpha}^\bullet\ : \ {\cal{A}}_0\longrightarrow \mathbb{K}$ is given by   $\hat{\alpha}^\bullet(\varphi)=i(\varphi)$, the diameter of the support of $\varphi$. We thus have that  $\hat{\alpha}^\bullet(\varphi_{\varepsilon})=  i(\varphi)\varepsilon $. A part from these technicalities, definitions and results mimic   those  from the previous sections.  Details can be found in the references mentioned at the end of this section.   Here  $\mathbb{K}\in \{\mathbb{R},\mathbb{C}\}$.

   Given $x\in\widetilde{\mathbb K}_f$   set $A(x):=\{r\in\mathbb R:\alpha^\bullet_r x\approx0\}$ and define the valuation of $x$ as $V(x)=sup(A(x))$. We have  that $D:\widetilde{\mathbb K}_f\times \widetilde{\mathbb K}_f\rightarrow \mathbb R_+$ defined by $D(x,y)=e^{-V(x-y)}$ is a translation invariant  ultra-metric on $\widetilde{\mathbb K}_f$.   It determines a uniform structure on  $\widetilde{\mathbb K}_f$,  the sharp uniform structure,   and the topology resulting from $D$ is   the sharp topology. Setting  $\|x\|=D(x,0) \  x\in \widetilde{\mathbb K}_f$,   the distance between   $x, y \in \widetilde{\mathbb K}_f $ is given by $D(x,y)=\|x-y\|$. This distance extends  to $\widetilde{\mathbb K}_f^n $ in the obvious way.

Let $\Omega \subset \mathbb{R}^n$ and let $\ (\Omega_l)_l$ be an exhaustion of $\Omega$. For $f\in \mathcal G_f(\Omega)$, define the family of sets $A_{lp}(f):=\{r\in\mathbb R:\alpha^\bullet_r\|f\|_{lp}\approx 0\}$ and define the family of pseudo-valuations $V_{lp}(f):=sup (A_{lp}(f))$. It follows that $D_{lp}: \mathcal G_f(\Omega)\times \mathcal G_f(\Omega)\rightarrow \mathbb R_+$ defined by $D_{lp}(x,y)=e^{-V_{lp}(x-y)}$ is an family of pseudo-ultrametrics on $\mathcal G_f(\Omega)$, where $l,p\in\mathbb N$ denote the index of the exhaustion $(\Omega_l)$ and the order of the derivative of $f$, respectively.  The  uniform structure it determines on  $\mathcal G_f(\Omega)$ is called the sharp uniform structure  and the resulting   topology is called  the sharp topology.    If $ x=[(x_\varphi)_\varphi]\in\widetilde{\mathbb K}_f$ we set  $  |x|=[(|x_\varphi|)_\varphi]$. If $x_0\in\widetilde{\mathbb K}_f$ and $r\in\mathbb R$, let  $V_r[x_0]:=\{x\in\widetilde{\mathbb K}_f:|x-x_0|\leq\alpha_r^\bullet\} \ \ \mbox{and} \ \ \mathcal B_f:=\{V_r[0]:r\in\mathbb R\}$. Similarly, for every $f\in \mathcal G_f(\Omega)$, and fixed $\beta\in \mathbb N^n_0$ and $l\in\mathbb N$,  let $\|f\|_{\beta,l}:=[(\|\partial^\beta f_\varphi(\cdot)\|_l)_{\varphi\in {\cal{A}}_0}]$, where $\|\cdot\|_l$ denotes the supremum norm over $\Omega_l$. Finally, for  $f_0\in \mathcal G_f(\Omega)$ and $r\in\mathbb R$, define  $W_{l,r}^\beta[f_0]:=\{f\in\mathcal G_f(\Omega):\|f-f_0\|_{\sigma,l}\leq\alpha_r^\bullet \ \ \forall \ \sigma\leq\beta\}$ and 
  $\mathcal B_{f\Omega}:=\{W_{l,r}^\beta[0]:\beta\in\mathbb N^n_0, \ l\in\mathbb N \ e \ r\in\mathbb R\}$.

\begin{theorem} $\ $

\begin{enumerate}
\item $\mathcal B_f$ is a system of 0-neighborhoods of   $\ \widetilde{\mathbb K}_f$  that induces a topology compatible with the ring structure of $\ \widetilde{\mathbb K}_f$ and coincides with the sharp topology.

\item$\mathcal B_{f\Omega}$ is a system of 0-neighborhoods of $\mathcal G_f(\Omega)$ that induces a topology compatible with the ring structure of $\mathcal G_f(\Omega)$ and coincides with the sharp topology. 
\end{enumerate}
Moreover $\ \widetilde{\mathbb K}_f$ is precisely the set of constant generalized functions on $\mathcal{ G}_f(\Omega)$ and the topology of the former is induced by that of the latter.

\end{theorem}



 
\begin{definition} {\bf[Hypernatural Numbers and Hypersequences]} 

\begin{enumerate}
\item   The subset $\widetilde{\mathbb N}:=\{n=[(n_\varphi)_{\varphi\in A_0}]\in \widetilde{\mathbb R}: \ n_\varphi\in \mathbb N \ \forall \ \varphi\in A_0\}$ is called the set of hypernatural numbers. In short,   hypernaturals are generalized numbers with representatives in ${\cal{E}}_{M_f}(\mathbb N)$.

\item   A hypersequence is a function $f:\widetilde{\mathbb N}\rightarrow \mathcal G_f$ and will be denoted by $(f_n)_{n\in\widetilde{\mathbb N}}$.
 
\item A hypersequence $(f_n)_{n\in\widetilde{\mathbb N}}$ converges to $f\in\mathcal G_f$ if for any neighborhood of zero $W_{l,r}^\beta[0]$ there exists $n_0\in\widetilde{\mathbb N}$ such that if $n\geq n_0$ then $(f_n-f)\in W_{l,r}^\beta[0]$.
\end{enumerate}

\end{definition}

If $f(\widetilde{\mathbb N})\subset \widetilde{\mathbb K}_f$ then the equality $\widetilde{\mathbb K}_f\cap W_{l,r}^\beta[0]=V_r[0]$ implies that $f_n\rightarrow f$ if under the same conditions of the above definition we have $(f_n-f)\in V_r[0]$.   A hypersequence $(f_n)_{n\in\widetilde{\mathbb N}}$ is \textit{Cauchy} if for any neighborhood of zero $W_{l,r}^\beta[0]$ there exists $n_0\in\widetilde{\mathbb N}$ such that if $n,m\geq n_0$ then $(f_n-f_m)\in W_{l,r}^\beta[0]$. Similarly, we can replace the condition $(f_n-f_m)\in W_{l,r}^\beta[0]$ by $(f_n-f_m)\in V_r[0]$ if $f(\widetilde{\mathbb N})\subset \widetilde{\mathbb K}_f$. Since $\mathcal G_f$ is complete, we have that $(f_n)_{n\in\widetilde{\mathbb N}}$ is convergent if, and only if, it is Cauchy. Note that $\widetilde{\mathbb{N}}$ is uncountable and not totally ordered. Hence  converging via  a hypersequence is  a matter of choices but all leading to the same point. Recall that in our model $\mathbb{K}^n$ is embedded into $\widetilde{\mathbb{K}}_f^n$ as a grid of equidistant  points, this common distance being equal to $1$. Hence the concept of hypersequences  helps to understand how  physical reality behaves when distances are smaller than  a certain constant. Here this distance is $1$, which is just a  symbol and  thus  replaceable   by any other constant,  multiplying the ultrametric by this constant.


    \begin{example}
The  sequence $(\frac{1}{n})_{n\in \mathbb{N}}$ does not converge  in $\widetilde{\mathbb{K}}_f$ but the hypersequence $(\frac{1}{n})_{n\in \widetilde{\mathbb N}}$ does converges to $0$ in $\widetilde{\mathbb{K}}_f$. In fact, for an arbitrarily chosen $r>0$, it suffices to take $n_0=[(\lfloor i(\varphi)^{-r}+1\rfloor)_\varphi]$. Then, if $n=[(n_\varphi)_\varphi]>n_0$, we have $\frac{1}{n_\varphi}<\frac{1}{\lfloor i(\varphi)^{-r}+1\rfloor}\leq\frac{1}{i(\varphi)^{-r}}$ for all $\varphi\in {\cal{A}}_q(1)$ with $q>p$ for some $p\in\mathbb N$. Therefore, $|\frac{1}{n}-0|<\alpha^\bullet_r$ and thus $\frac{1}{n}\in V_r[0].$
\end{example}

Let $(B_\varphi)_{\varphi\in {\cal{A}}_0}$ be a net of subsets of $\mathbb{K}^n$. The internal set  $B=[(B_\varphi)_\varphi]$ is the subset $B\subset \widetilde{\mathbb{K}}_f ^n$ whose points $p\in B$ have some representative $(\hat{p}_{\varphi})$ with $\hat{p}_{\varphi}\in B_{\varphi}$, $\forall \varphi\in {\cal{A}}_0$. Since points of $\widetilde{\mathbb{K}}_f ^n$ must be thought of as germs, in the full environments, this  translate into the following: there exist $q\in \mathbb{N}$ such that $\forall \varphi\in {\cal{A}}_q$, there exists $\eta_\varphi\in I$ such that $\hat{p}_{\varphi_{\varepsilon}}\in B_{\varphi_{\varepsilon}}$, $\forall \varepsilon <\eta_\varphi$.

\begin{definition}
Let $B=[(B_\varphi)_\varphi]$, $C=[(C_\varphi)_\varphi]$ be internal sets in $\mathcal G_f$. An application  $T:B\rightarrow C$ is a map represented by a moderate  net  $(T_\varphi:B_\varphi\rightarrow C_\varphi)_{\varphi\in {\cal{A}}_0}$ such that $T(f)=[(T_\varphi(f_\varphi))_\varphi]$ for all $f=[(f_\varphi)_\varphi]\in B$.
\end{definition}

\begin{definition}{\bf{[Contraction]}}
Let $B=[(B_\varphi)_\varphi]\subset \mathcal G_f$ be an internal set, and let $T:B\rightarrow B$ be an application with representative $(T_\varphi:B_\varphi\rightarrow B_\varphi)_{\varphi\in {\cal{A}}_0}$. We say that $T$ is a contraction if there exist $L\in \widetilde{\mathbb R}^*_{+f}$ and $\lambda\in (0,1)$ such that $L<\lambda$ and $\|T(f)-T(g)\|_l\leq L\|f-g\|_l$ for all $l\in\mathbb N$.
\end{definition}
    

Given  $n=[(n_\varphi)_\varphi]\in\widetilde{\mathbb N}$,  the map $T^n:B\rightarrow B$ with representative $(T^{n_\varphi}_\varphi : B_\varphi\rightarrow B_\varphi)_{\varphi\in {\cal{A}}_0}$, where $n_\varphi$ compositions of $T_\varphi$ are understood, if well defined  is a contraction since  $\|T^n(f)-T^n(g)\|_l\leq L^n\| f-g\|_l\leq\lambda^n\| f-g\|_l$ for all $l\in\mathbb N$. If it is necessary to specify $\lambda$, we shall say that $T$ is a $\lambda$-contraction. Let   $\Omega\subset\mathbb R^n$, $l,p\in\mathbb N$ and  consider the ultrametric in $\mathcal G_f(\Omega)$  $D: \mathcal G_f(\Omega)\times \mathcal G_f(\Omega)\rightarrow \mathbb R_+$ defined by  $$D(f,g):=sup\left\{\frac{2\cdot D_{ll}(f,g)}{1+D_{ll}(f,g)}: \ l\in\mathbb N \right\}$$


Let $r_0<2$ and $f\in B_{r_0}(0)$. Then it follows that $\frac{2\cdot D_{ll}(f,g)}{1+D_{ll}(f,g)}\leq r_0, \ \forall \ l\in \mathbb{N}$ and hence $V_{ll} (f,g)\geq ln(\frac{2-r_0}{r_0}),\ \forall\  l\in \mathbb{N}$. Take $d_0$ such that $\ln (d_0)\leq ln(\frac{2-r_0}{r_0})$, then $\| f\|_{l} \leq \alpha^{\bullet}_{d_0}$, $\forall\ l\in \mathbb{N}$.    Now let  $\lambda\in\  ]0,1[$ be fixed and suppose that we want $\lambda^{n_0}\in V_t(0)$, where $n_0=[(n_{0\varphi})]$ and $t>0$.  For this to occur one must have $\lambda^{n_{0\varphi} }<i(\varphi)^t$. From this it follows that $n_{0\varphi}>(\frac{-t}{|\ln(\lambda)|})\cdot \ln(i(\varphi))$. Hence we may take $n_{0\varphi}=2\cdot \biggl\lfloor\frac{-t}{|\ln(\lambda)|} \cdot \ln(i(\varphi)))\biggr\rfloor$. It is clear that $n_0<\alpha^{\bullet}_{-1.1}$ and therefore  is moderate. For any $m>n_0$ we have $\lambda^m<\lambda^{n_0}\in V_t(0)$ and thus  $\lambda^m \in V_t(0)$ . This proves that the hypersequence $(\lambda^n)$ converges to $0$.  We shall use these two facts  in the proof of our next theorem.

 \begin{theorem}{\bf{[Fixed Point Theorem]}}

Let $\Omega\subset\mathbb R^N$, $A=[(A_\varphi)_\varphi]\subset B_{r_0}(0)\cap\mathcal G_f(\Omega)$, $r_0<2$, be an internal set, and $T:A\rightarrow A$ be given with representative $(T_\varphi:A_\varphi\rightarrow A_\varphi)_{\varphi\in {\cal{A}}_0}$. If there exists $k=[(k_\varphi)_\varphi]\in\widetilde{\mathbb N}$ such that   $T^k=(T_\varphi^{k_\varphi})$ is a $\lambda$-contraction, then $T$ is well defined,     continuous and has a unique fixed point in $  A$.

\end{theorem}

\begin{proof} We start by noting that $B$ is a complete topological space. This follows because internal sets are closed and ${\cal{G}}_f(\Omega)$ is a complete ultrametric space.   The  Lipschitz condition implies that $T$   is  well-defined and continuous.  For   $g_0\in A$,  we shall prove that  the hypersequence $(T^n(g_0))_{n\in\widetilde{\mathbb N}}$ converges to a point $f_0\in A$. To achieve this it is enough  to prove that $(T^n(g_0))_{n\in\widetilde{\mathbb N}}$ is a Cauchy hypersequence. Consider the basic neighborhood  $W_{l,r}^l[0]\in {\cal{B}}_f$.   Since $0<\lambda < 1$ and $A\subset B_{r_0}(0)$,  and using what we mentioned in the paragraph before the theorem,    we may  choose  $n_0\in \widetilde{\mathbb N}$  and $d_0$ such that $\lambda^{n_0}\alpha^\bullet_d\in V_{4^lr}[0]$.  Set  $r_1=4^{l-1}r$ and let  take $n,s>n_0$. Writing $n=n_0+p$ and $s=n_0+q$ we have that $\|T^p(g_0)-T^q(g_0)\|_l\leq\alpha^\bullet_{d_0}$ and thus  $\|T^n(g_0)-T^s(g_0)\|_l = \|T^{n_0}(T^p(g_0))-T^{n_0}(T^q(g_0))\|_l 
\leq    \lambda^{n_0} \|T^p(g_0)-T^q(g_0)\|_l 
\leq  \lambda^{n_0}\alpha^\bullet_{d_0}\in V_{4r_1}[0] $.  This proves that $F:= T^n(g_0)-T^s(g_0)\in W^0_{l,4r_1}[0]$.  Consider the embedding   $\kappa : \mathcal G_f(\Omega)\longrightarrow\mathcal {\cal{C}}^{\infty}(\widetilde{\Omega}_{cf},\widetilde{\mathbb K}_f)$ and identify  $F$   with $\kappa(F)$. Without loss of generality, we may suppose that $N=1$. Since $F$ is differentiable in $\widetilde{\Omega}_{cf}$, it follows that $$\lim\limits_{r_1\rightarrow +\infty}\frac{F(x+\alpha^\bullet_{2r_1})-F(x)}{\alpha^\bullet_{2r_1}}=F'(x),$$  for all $x\in\widetilde{\Omega}_{cf}$. Therefore,

\begin{equation}
\lim\limits_{r_1\rightarrow +\infty}\frac{|F(x+\alpha^\bullet_{2r_1})-F(x)|}{\alpha^\bullet_{2r_1}}=|F'(x)|. \label{ptofixo1}
\end{equation}

\noindent Since $F\in W^0_{l,4r_1}[0]$, we have that  $$|F(x+\alpha^\bullet_{2r_1})-F(x)|\leq |F(x+\alpha^\bullet_{2r_1})|+|F(x)|\leq 2\alpha^\bullet_{4r_1}$$ 

\noindent and therefore

\begin{equation}
\frac{|F(x+\alpha^\bullet_{2r_1})-F(x)|}{\alpha^\bullet_{2r_1}}\leq 2\alpha^\bullet_{2r_1}
\label{ptofixo2}
\end{equation}

\noindent  for all $x\in\widetilde{\Omega}_{cf}$. From \ref{ptofixo1} and \ref{ptofixo2}, it follows that $|F'(x)|\leq 2\alpha^\bullet_{2r_1}<\alpha^\bullet_{r_1}$ for all $x\in\widetilde{\Omega}_{cf}$, which gives us that $F\in W^1_{l,r_1}[0]$. Since $F$ is a ${\cal{C}}^{\infty}-$function, we may repeat the same process for its derivates up to order $l$  concluding  that $F\in W^{l}_{l,r}[0]$ and thus,   $(T^n(g_0))_{n\in\widetilde{\mathbb N}}$ is a Cauchy hypersequence.  Uniqueness of the fixed point being  obvious,  completes the proof.
 \end{proof}
 
 
From the   induction step given in the proof of the theorem we deduce  the  topological tool which we coin   the {\it  Down Sequencing Argument} and  shorten it  to   DSA.  The DSA  ensures that moderation and nullity at level 0 imply nullity.

\begin{theorem}{\bf{[Down Sequencing Argument]}}
Let $f\in\mathcal G_f(\Omega)$,    $f\in W^0_{l,4^kr}[0]$ with $r>0$.  Then $f\in W^{k}_{l,r}[0]$, i.e.,  $W^0_{l,4^kr}[0]\subset W^{k}_{l,r}[0]$.
\end{theorem}

The notion of association has long been seen and used as an algebraic notion. But here we observe  that  it is in fact  a topological notion which implies the following.

\begin{theorem}
Let $\Omega$ be an open subset of $\mathbb{R}^n$. Then ${\cal{D}}^{\prime}(\Omega)$ embeds as a  discrete grid in ${\cal{G}}(\Omega)$. Moreover, ${\cal{C}}^{\infty}(\Omega)$   embeds as a grid of equidistant  points into ${\cal{G}}_f(\Omega)$.
\end{theorem}

 The theorem shows  that classical solutions to differential equations are rare. This can be used as a tool to obtain classical solutions from a generalized solution. References on which this section is based are \cite{aragona1991intrinsic, aragona2005discontinuous, aragonamembranas, aragona2009natural, https://doi.org/10.48550/arxiv.1706.02810, garetto2011hilbert, Giordiano1, juriaans2022fixed, ku}.

\section{Fundamentals of Generalized Geometry}

In this section we extend the constructions made for  subset  $\Omega \subset \mathbb{K}^n$ to abstract manifolds. This is done in  the full milieus.  Starting with a  Riemannian sub-manifold  $M \subset \mathbb{K}^n$   we construct    a generalized sub-manifold   $M^*\subset \widetilde{\mathbb{K}}^n$, whose dimensions over the respective underlying reals  are equal.  The construction is such that $M$ is the shadow of $M^*$, or, in the notation to be introduced in this section, $ssupp(M^*)=M$.  It is here that we piece the puzzle   letting tools and concepts    developed by several prominent researchers fall into place.  In particular, Classical Space-Time can be embedded into Generalized Space-Time and thus  making available tools   that can  be used to explain phenomena in physical reality.  We begin recalling some definitions and preparing the needed full environment machinery we shall be using.

\begin{definition}
Let $M$ be a non-empty set. A ${\cal{C}}^\infty$ $\mathcal G_f$-atlas of dimension $n$ on $M$ is a family $\mathcal A={(U_\lambda,\varphi_\lambda)}_{\lambda\in\Lambda}$, where $\Lambda$ is an index set, that satisfies the following conditions:
\begin{enumerate}
\item For every index $\lambda\in\Lambda$, the map $\varphi_\lambda:U_\lambda\rightarrow \widetilde{\mathbb R}_f^n$ is a bijection between the non-empty open subset $U_\lambda\subset M$ and the open subset $\varphi_\lambda(U_\lambda) \subset\widetilde{\mathbb R}_f^n$;
\item $M=\displaystyle\bigcup\limits_{\lambda\in\Lambda}U_\lambda$;
\item For every pair $\alpha,\beta\in\Lambda$ with $U_{\alpha\beta}=U_\alpha\cap U_\beta\neq\emptyset$, the subsets $\varphi_\alpha(U_{\alpha\beta})$ and $\varphi_\beta(U_{\alpha\beta})$ are open subsets contained in $\widetilde{\mathbb R}_f^n$ such that $\varphi_\beta\circ\varphi_\alpha^{-1}:\varphi_\alpha(U_{\alpha\beta})\rightarrow\varphi_\beta(U_{\alpha\beta})$ is a ${\cal{C}}^\infty$ diffeomorphism.
\end{enumerate}
\end{definition}

The pair $(U_\alpha,\varphi_\alpha)$ is called a local chart (or a coordinate system) of $M$. If $U\subset M$ and $\varphi:U\rightarrow\varphi(U)$ is a homeomorphism, where $\varphi(U)$ is an open subset of $\widetilde{\mathbb R}^n_f$, then the pair $(U,\varphi)$ is said to be compatible with $\mathcal A$ if for every pair $(U_\lambda,\varphi_\lambda)\in\mathcal A$, with $W_\lambda=U\cap U_\lambda\neq\emptyset$, we have that $\varphi\circ\varphi^{-1}_\lambda:\varphi_\lambda(W_\lambda)\rightarrow\varphi(W_\lambda)$ is a ${\cal{C}}^\infty$ diffeomorphism, where $\varphi_\lambda(W_\lambda)$ and $\varphi(W_\lambda)$ are open subsets of $\widetilde{\mathbb R}^n_f$. By  Zorn's Lemma,   there exists a unique maximal ${\cal{C}}^\infty$ $\mathcal G$-atlas $\mathcal A^\ast\supset {\cal{A}}$ of dimension $n$ on $M$, namely, the atlas$$(U,\varphi)\in\mathcal A^\ast, \ \mbox{if and only if} \  (U,\varphi) \ \mbox{is}  \  \mbox{compatible with} \ \mathcal A.$$ 


	A generalized   manifold in the full environment, or a $\mathcal{G}_f$-manifold, is a set $M$ equipped with a $\mathcal{G}_f$-atlas. A maximal $\mathcal{G}_f$-atlas of the $\mathcal G_f$-manifold $M$ is called a $\mathcal{G}_f$-differential structure on  $M$.  The topology of $M$ is the one that makes all local charts simultaneously homeomorphisms. If the context is clear, we shall  omit the prefix $\mathcal{G}_f$.  The proof of our next result can be found in \cite{juriaans2022fixed}.

\begin{theorem}{\bf{[Dimension Invariance]}}
Let $M$ be a $\mathcal G_f$-manifold. Then the dimension of a $\mathcal G_f$-atlas $\mathcal A$ is constant in each connected component of $M$.

\end{theorem}

 The notion of support of a generalized point was introduced in the context of the simplified algebra. Given a point  $p=[(p_\varepsilon)_{\varepsilon}]\in\widetilde{\mathbb K}^n$ one can see easily that 
 
 $$supp(p)=\bigcap\limits_{\eta\in I}\overline{\{p_{\varepsilon}\ : \ \varepsilon\in I_\eta\}}$$
 
As we already saw, the definition can be translated in terms of idempotents as follow: $q\in supp(p)$ if and and only if these exists an idempotent $e\in{\cal{B}}(\widetilde{\mathbb{R}})$ such that $e\cdot p \approx e\cdot q$. In the full environment we have that  $x\in\widetilde{\mathbb K}_f$ is associated with 0,  $x\approx0$, if for a representative  $(x_\varphi)_\varphi$ of $x$, there exists $p\in\mathbb N$ such that $x_{\varphi_\varepsilon}\rightarrow0$ when $\varepsilon \rightarrow 0$, for all $\varphi\in\mathcal A_p(1)$. The identification $x_1\approx x_2$ is equivalent to $x_1-x_2\approx0$. If there exists $x_0\in\mathbb K$ such that $x\approx x_0$, we say that $x_0$ is the associated number (or shadow) of $x$. This   definition  readily   extends to $\widetilde{\mathbb K}_f^n$ in an obvious way.

Let $p=[(p_\varphi)_{\varphi}]\in\widetilde{\mathbb K}_f^n$. The {\it support} of $p$ is defined to be the subset $$supp_f(p):=\{q\in\mathbb K^n: \  \forall \ m\in\mathbb{N}, \  \ \exists \ \varphi\in{\cal{A}}_m \ \mbox{and} \  \exists \ \varepsilon_l\rightarrow0 \ \mbox{such that} \ p_{\varphi_{\varepsilon_{_l}}}\rightarrow q\}$$

 \noindent  and the {\it essential support} of $p$ is defined as

  $$ssupp_f(p)=\bigcap\limits_{m\in \mathbb{N}}\overline{\{p_{\varphi}\ : \ \varphi\in  {\cal{A}}_m\}}$$

\noindent Clearly $supp_f(p)\subset ssupp_f(p)$.  Algebraically we have:   $q_0 \in supp_f(p)$   if and only if there exists  $e\in {\cal{{B}}(\widetilde{\mathbb{R}}}_f)$ such  $e \cdot  p \approx e \cdot q_0$.  From now on, subscripts will be omitted when referring to the support of a point.  If $M$ is a Riemannian manifold and $ p\in \widetilde{M}_{cf}$ then       $ssupp(p)$ is a compact subset of $M$. By construction, $ssupp(\widetilde{M}_{cf})=\overline{M}$.  If $r>0$, then $supp(\alpha^\bullet_r)=\{0\}$.   More generally, if $p\in B_1(0)$ then $supp(p)=\{0\}$, since in this case we have $p\approx0$ (see \cite[Lemma 2.1]{aragona2013algebraic}).   If $p=[(\sin(\alpha_{-1}^\bullet)]$, then $supp(p)=[-1,1]$.  For  $x=[(x_\varphi)_\varphi]\in \widetilde{\mathbb{ R}}^n_f$, write $\|x\|_2=[(\|x_{\varphi}\|_{\mathbb{R}^n})_\varphi ]$, but note that any other norm of $\mathbb{R}^n$ will do.  

 If $(B_\varphi)_{\varphi\in\mathcal {\cal{A}}_0}$ is a net of subsets of $\mathbb R^n$, then the internal set generated by $(B_\varphi)_{\varphi\in{\cal{A}}_0}$ is $[B_\varphi]:=\{x\in \widetilde{\mathbb R}_f^n:\exists \ (x_\varphi)_\varphi \ representative  \ of \ x, \ \exists \ k\in\mathbb N,\ \exists \ \eta_{_\varphi}\in I \ $,  $ \mbox{such that}  \ x_{\varphi_\varepsilon}\in B_{\varphi_\varepsilon} \ \forall \ \varphi\in{\cal{A}}_k\  \mbox{and} \ \varepsilon\in I_{\eta_{\varphi}} \}$.  Internal sets generalize the notion of membranes  and are closed in the sharp topology, $\tau$.  Strong internal sets were first introduced in the simple environments and  are part of the puzzle.

\begin{definition}{\bf{[Strong Internal Sets]}}$\ $
\begin{enumerate}
\item Let $(B_\varphi)_{\varphi\in{\cal{A}}_0}$ be a net of subsets of $ \mathbb R^n$ and let  $(x_\varphi)_{\varphi\in{\cal{A}}_0}$ be a moderated net of points in $\mathbb R^n$. We say that $(x_\varphi)_{\varphi\in{\cal{A}}_0}$ belongs to $(B_\varphi)_{\varphi\in{\cal{A}}_0}$ $\tau$-strongly, and we write $x_\varphi\in_\tau B_\varphi$, if 

\begin{enumerate}

\item There exists $k\in\mathbb N$ such that $x_{\varphi_\varepsilon}\in B_{\varphi_\varepsilon}$ for all $\varphi\in {\cal{A}}_k$ and $\varepsilon$ sufficiently small.

\item If $[(y_\varphi)_\varphi]=[(x_\varphi)_\varphi]$ then there exists $p\in\mathbb N$, $p\geq k$, such that $y_{\varphi_\varepsilon}\in B_{\varphi_\varepsilon}$,  for all $\varphi\in{\cal{A}}_p$ and $\varepsilon$ sufficiently small.

\end{enumerate}

\item   Let $(B_\varphi)_{\varphi\in{\cal{A}}_0}$ be a net of subsets of $\mathbb R^n$.  The strong internal set with respect to $\tau$ ($\tau$-strong internal) generated by $(B_\varphi)_{\varphi\in{\cal{A}}_0}$ is denoted and defined by: $\langle B_\varphi\rangle_\tau:=\{[(x_\varphi)_\varphi]\in \widetilde{\mathbb R}^n_f: x_\varphi\in_\tau B_\varphi\}$.  We  shall omit  $\tau$ in the notation $\langle B_\varphi\rangle_\tau$, call it a strong internal set and  also replace the notation $\in_\tau$ by $\in_\varphi$. 

\end{enumerate}
\end{definition}

  It is straightforward to verify that  $\widetilde{\mathbb R}_f=\langle (-e^{1/i(\varphi)}, e^{1/i(\varphi)})_{\varphi\in {\cal{A}}_0}\rangle$ and that  $\langle A_\varphi\rangle\cap\langle B_\varphi\rangle=\langle A_\varphi\cap B_\varphi\rangle$.

\begin{theorem}
Let $(A_\varphi)_{\varphi\in{\cal{A}}_0}$ be a net of subsets of $\mathbb R^n$, and let  $(x_\varphi)_{\varphi\in{\cal{A}}_0}$  be a moderated net of points in $\mathbb R^n$. Then, 
 $x\in \langle A_\varphi \rangle \Longleftrightarrow $   $dist(x,A^c):=[(d(x_\varphi,A_\varphi^c))_\varphi]\in Inv(\widetilde{\mathbb R}_f)$. 
 \label{internoforte}
\end{theorem}

\begin{proof}
Consider $x_\varphi\in_\varphi A_\varphi$ and suppose that $\forall \ p\in\mathbb N \ \exists \ \varphi\in{\cal{A}}_p$ and $\varepsilon_k\rightarrow0$ such that $d(x_{\varphi_{_{\varepsilon_k}}},A_{\varphi_{_{\varepsilon_k}}}^c)\leq i(\varphi)^k\varepsilon_k^k, \ \forall \ k\in\mathbb N$. Then, for each $k\in\mathbb N$, we can choose $y_k\in A_{\varphi_{{\varepsilon_k}}}^c$, such that $\|y_k-x_{\varphi_{_{\varepsilon_k}}}\|<2i(\varphi)^k\varepsilon_k^k, \ \forall \ k\in\mathbb N   $.


Choose $(y_\varphi)_\varphi$ equivalent to $(x_\varphi)_\varphi$, so that for such $\varphi's$ that we assume exist, we have $y_{\varphi_{_{\varepsilon_k}}}=y_k, \ \forall \ k\in\mathbb N$. Thus, $y_{\varphi_{{\varepsilon_k}}}\notin A_{\varphi_{{\varepsilon_k}}}$ $\forall \ k\in\mathbb N$, a contradiction. On the other hand, if the hypothesis of the converse implication holds, then $x_{\varphi_\varepsilon}\in A_{\varphi_\varepsilon}$, $\forall \ \varphi\in{\cal{A}}_p$ and $\varepsilon$ sufficiently small. Moreover, if $(y_\varphi)_\varphi$ is equivalent to $(x_\varphi)_\varphi$, then $d(y_{\varphi_\varepsilon},A_{\varphi_\varepsilon}^c)>\frac{1}{2}\alpha^\bullet_r(\varphi_\varepsilon)$ and therefore, $y_{\varphi_\varepsilon}\in A_{\varphi_\varepsilon}, \ \forall \ \varphi\in{\cal{A}}_p \ \mbox{and}\   \varepsilon \ \mbox{small  enough}$. Using \cite[Theorem 3.18]{aragona2013algebraic}, the result follows.
\end{proof}

\begin{corollary}
$\langle A_\varphi\rangle$ is open in the sharp topology. \label{intforteaberto}
\end{corollary}

Let $M$ be a connected sub-manifold of dimension $n$ in $\mathbb R^N$, and let ${(U_\alpha,\phi_\alpha),\alpha\in\Lambda}$ be an atlas of $M$. Suppose that $\forall \ \alpha\in\Lambda$, $\phi_\alpha(U_\alpha)=\Omega_0=B_r(0)$, for some $r>0$. Consider $\widetilde{M}_{c,f}\subset\widetilde{\mathbb R}^N_{c,f}\subset \widetilde{\mathbb R}^N_{f}$ the set of compactly supported points defined by $M$. Define  $\widetilde{\Lambda}:=\{\lambda:{\cal{A}}_0\rightarrow\Lambda\}={\cal{F}}({\cal{A}}_0,\Lambda)$.   Each $\lambda\in\widetilde{\Lambda}$ is associated with a net  $(\lambda_\varphi)_{\varphi\in{\cal{A}}_0}$  of elements from $\Lambda$, where $\lambda_\varphi=\lambda(\varphi)$. For $\lambda\in\widetilde{\Lambda}$, define $U_\lambda:=\langle U_{\lambda_\varphi}\rangle\subset \widetilde{\mathbb R}^N_{c,f}$ and define  
$$
\begin{aligned}
 \phi_{\lambda}: U_{\lambda} & \longrightarrow \widetilde{\mathbb R}^n_{c,f} \\ [(p_\varphi)_\varphi] & \longmapsto [(\phi_{\lambda_\varphi}(p_\varphi))_\varphi].
 \end{aligned}
 $$



 \begin{theorem}
Suppose that for each $\alpha \in \Lambda$, the map $\phi_\alpha$ and its inverse are Lipschitz. Then $\widetilde{M}_{c,f}$,  with the induced topology,  is a $\mathcal{G}_f$-manifold of $\widetilde{\mathbb{R}}^N_f$ of dimension $n$.
\end{theorem}

\begin{proof}
Let  $p=[(p_\varphi)_\varphi]\in \widetilde{M}_{c,f}$. We shall construct a local chart containing  $p$.  In  fact,  since $ssupp(p)$ is a compact subset of $M$, there exists a finite subset $I_0\subset \Lambda$ such that  $ssupp(p)\subset\bigcup\limits_{\alpha \in I_{0}} U_{\alpha}$. Let $\delta$ be the Lebesgue number for this covering. Choose a finite number of points $q_i\in ssupp(p)$ such that $ssupp(p)\subset\bigcup\limits_{1\leq i\leq l} B_{\delta_1}(q_i)$, where $\delta_1<\delta/4$. Starting with $q_1$, define $\lambda\in\widetilde{\Lambda}$ by $\lambda_{\varphi}:=\alpha_{q_{1}}$ where $\alpha_{q_{1}}$ is chosen such that $B_{\delta_{1}}(q_{1}) \subset U_{\alpha_{q_{1}}}$ and $p_\varphi\in B_{\delta_{1}}(q_{1})$. For $\lambda_\varphi$ not yet defined, continue this process defining $\lambda_{\varphi}:=\alpha_{q_{2}}$ where $\alpha_{q_{2}}$ is chosen such that $B_{\delta_{1}}(q_{2}) \subset U_{\alpha_{q_{2}}}$ and $p_\varphi\in B_{\delta_{1}}(q_{2})$ until completing it by defining $\lambda_{\varphi}:=\alpha_{q_{l}}$. At the end of this process, $\lambda$ will be well defined. In fact, if it were not so, then $\forall \ m\in\mathbb N$, there should exist a sequence   $(\varphi_k)\in {\cal{A}}_{m_k}, \ m_k\uparrow\infty$ and a sequence $\varepsilon_k\rightarrow0$ such that $p_{(\varphi_k)_{{\varepsilon_{k}}}}\notin B_{\delta_{1}}(q_{i})$ for any   $1\leq i\leq l$, and for sufficiently large values of $k$. This is a contradiction, since  the sequence $(p_{(\varphi_k)_{\varepsilon_{k}}})$ has an accumulation point in $ssupp(p)$ and the latter is covert by the balls. Therefore, $\lambda$ is well-defined and has  finite image. Furthermore, $p\in U_\lambda$. To see this, we need to show that $p_\varphi\in_\varphi U_{\lambda_\varphi}$. But this follows directly from Theorem~\ref{internoforte} and the definition of $\lambda$, since by the definition of $\lambda$, we have $p_\varphi\in B_{\delta_{1}}(q_{_i}) \subset U_{\alpha_{q_{_i}}}$ for some $i$, by Theorem \ref{internoforte} and appropriate choice of $\delta_1$, we have that 
$[(d(p_\varphi,U_{\lambda_\varphi}^c))_\varphi]$ is invertible in $\widetilde{\mathbb R}_f$. 

We shall now prove that the family $ \{ ( U_{\lambda}, \phi_{\lambda}), \lambda \in {\widetilde{\Lambda}}\}$, with $\lambda$ of finite range, is a ${\cal{C}}^{\infty} $  $\mathcal G_f-$atlas of dimension $n$ for $\widetilde{M}_{c,f}$.  In fact, it is clear that $\phi_\lambda$ is well-defined for each $\lambda$, i.e., it does not depend on representatives and $\phi_\lambda(p)\in\langle\Omega_0\rangle$,  $p\in U_\lambda$. This follows directly from the Lipschitz condition of the charts of $M$ and their inverses, and from $\lambda$ having finite range. It also follows that each $\phi_\lambda$ is an isometry with respect to the sharp topologies of $\widetilde{\mathbb R}^N_f$ and $\widetilde{\mathbb R}^n_f$, and is therefore bijective and continuous. By Corollary~\ref{intforteaberto}, we already have that each $U_\lambda$ and $\langle\Omega_0\rangle$ are open in the sharp topology. It is clear that $\widetilde{M}_{c,f}=\displaystyle\bigcup_{\lambda}U_\lambda$. 

Any change of coordinates $\phi_\beta\circ\phi_\lambda^{-1}$ is a homeomorphism that has a representative consisting of finitely many ${\cal{C}}^\infty$ diffeomorphisms that take values in a bounded subset of $\mathbb R^n$, so $\phi_\beta\circ\phi_\lambda^{-1}$ is a ${\cal{C}}^\infty$ diffeomorphism from the open set $\phi_\lambda(U_{\lambda\beta})$ to the open set $\phi_\beta(U_{\lambda\beta})$. Note that we can write $\phi_\beta\circ\phi_\lambda^{-1}$ as a finite interleaving
$$\phi_\beta\circ\phi_\lambda^{-1}=\sum\limits_{i}e_if_i$$ 

\noindent where each $f_i$ is a ${\cal{C}}^\infty$ diffeomorphism. Finally, for each $\lambda_\varphi$, there exists an open set $U^{\lambda_\varphi}$ in $\mathbb R^N$ such that $U_{\lambda_\varphi}=M\cap U^{\lambda_\varphi}$. Defining $U^\lambda=\langle U^{\lambda_\varphi}\rangle$, it follows that $U_\lambda=\widetilde{M}_{c,f}\cap U^\lambda$, with $U^\lambda$ an open subset of $\widetilde{\mathbb R}^N_f$. This proves that $\widetilde{M}_{c,f}$ has  the induced topology of $\widetilde{\mathbb R}^N_f$, and concludes the proof.
\end{proof}

\begin{corollary}
Each local chart $(U_\lambda,\phi_\lambda)$ is an isometry in the sharp topologies.
\end{corollary}

\begin{proof}
In fact, by lemma A.1 of \cite{vernaeve2011isomorphisms}, for each compact subset $K\subset M$, there exists $C>0$ such that $\|p-q\|\leq dist_M(p,q)\leq C\|p-q\|$, for all $p,q\in K$, where $dist_M$ is the Riemannian metric of $M$. This implies that the local charts are isometries in  the sharp topology, since,  from the above proof, each one is an isometry with respect to the sharp topologies of $\widetilde{\mathbb R}^N_f$ and $\widetilde{\mathbb R}^n_f$.
\end{proof}

If $\lambda$ is constant then $(U_{\lambda},\varphi) $  is called a {\it principal chart}. For $\lambda \in \widetilde{\Lambda}\subset \widetilde{{\mathbb{K}}}_f$ and $x \in B_1(0)$  define  $x\lambda(\varphi) = x(\varphi)\lambda(\varphi)$.  In particular, if  $x=e\in {\cal{B}}(\widetilde{\mathbb{R}}_f)$ then $e\lambda $  means  that when $e(\varphi)  = 0$  this index must  be omitted. With this notation, the proof of the theorem gives us the following corollary and thus, once again, making the connection with the notion of interleaving. 

\begin{corollary} Let $(U_{\lambda},\varphi)$  be a local chart with  $\lambda$ of finite range. There exist principal charts $(U_{\alpha_i},\varphi_i) , i = 1, k$  and a complete set of mutually orthogonal idempotents $e_1, \cdots , e_k\in{\cal{B}}(\widetilde{\mathbb{R}}_f)$ such that $ e_i\cdot U_{\lambda}\subset e_i\cdot  U_{\alpha_i}$   and $U_{\lambda} = \bigsqcup_i  U_{e_i\alpha_i}$, where $\bigsqcup$ should be understood as some kind of interleaving. 

 \end{corollary}


Let $M$ be a Riemannian manifold. By Whitney's Embedding Theorem $M$ may be embedded as sub-manifold in some Euclidean space. Applying the previous theorem, we set $M^* =\widetilde{M}_{cf}$. By construction we have that $ssupp(M^*)=M$.    By construction we also have that $M\subset M^*\subset \widetilde{M}$, the latter containing the infinities and the former, $M^*$,  containing infinitesimals.

\begin{corollary}
Let $M$ be a Riemannian manifold. There exists a ${\cal{G}}_f-$manifold $M^*$ such that $M$ is discretely embedded in $M^*$ and $ssupp(M^*)=M$.
\end{corollary}

Let $f\in {\cal{C}}^{\infty}(M,\mathbb{R})$ and define $\iota(f)$ on $M^{*}$ such that its local expression on the local chart  $(U_\lambda,\phi_\lambda)$ of $M^{*}$  is given by  $f\circ \phi_\lambda^{-1}([(p_{\varphi})_{\varphi}])=[(f\circ\varphi^{-1}_{\lambda(\varphi)}(p_{\varphi}))_{\varphi}]$.   Since $f$ is differentiable, the  local expressions of $\iota(f)$ are differentiable functions and hence, as defined following classical Differential Geometry, it follows that  $\iota(f)\in {\cal{C}}^{\infty}(M^*,\widetilde{\mathbb{R}}_f)$.  Consequently, 

$$\iota :  {\cal{C}}^{\infty}(M,\mathbb{R}) \longrightarrow {\cal{C}}^{\infty}(M^*, \widetilde{\mathbb{R}}_f)$$

\noindent is an $\mathbb{R}$-algebra monomorphism.  

Let $M$ be an orientable Riemannian manifold. We now relate our construction with the construction of \cite{gkos4, invariant}, where the authors define the basic space of generalized scaler fields on $M$  and  denote it by $\hat{{\cal{E}}}(M)$. The subset of moderate elements is denoted by $\hat{{\cal{E}}}_m(M)$ and the   algebra of generalized functions on $M$ is denoted by $\hat{{\cal{G}}}(M)$. This is done in the setting of the full algebras which happens to be our setting too (not the invariant though). We observe   that ${\cal{D}}(M)$ consists of  forms but this does not invalidate what comes next.   Our next result extends the  Embedding Theorem proved in  \cite{aragona2005discontinuous} (see previous sections).


\begin{theorem} {\bf{[Embedding Theorem]}} Let $M$ be an $n-$dimensional orientable  Riemannian manifold. There exists an $n-$dimensional ${\cal{G}}_f-$manifold $M^*$  and an algebra monomorphism 

$$\kappa  :\hat{{\cal{G}}}(M)\longrightarrow {\cal{C}}^{\infty}(M^*,\widetilde{\mathbb{R}}_f)$$

\noindent which commutes with derivation. Moreover, $ssupp(M^*)=M$ and  equations defined on $M$,  whose data have singularities or nonlinearities,    naturally extend to equations on $M^*$ and, on $M^*$,  these data become $\ {\cal{C}}^{\infty}-$functions.

\end{theorem}

We skip the proof since it can be obtained from what we just saw  and  results of \cite{gkos4, invariant} on $\hat{{\cal{E}}}_m(M)$. The restriction of $\kappa$ to ${\cal{C}}^{\infty}(M,\mathbb{R})$ is $\iota$.  Some classical results can be readily extended: A differentiable map   $g : M\longrightarrow N$   between manifolds induces an  algebra homomorfism $\Phi(g): {\cal{C}}^{\infty}(N^*,\widetilde{\mathbb{R}}_f)\longrightarrow {\cal{C}}^{\infty}(M^*,\widetilde{\mathbb{R}}_f)$. Another point in our approach is that the topologies of the algebras and the  ${\cal{G}}_f-$manifolds involved have  the same underlying reals,  namely $\widetilde{\mathbb{R}}_f$. Note however that in our approach  some infinitesimals live in   $M$,   but  infinities  live in $\widetilde{M}$.   Function which are infinities do live in the algebras constructed.

We established that in  the setting of Colombeau algebras on manifolds the perspective  can be reduced to a  setting very much alike  to the settings of  Classical Differential Geometry which we coined Generalized Differential Geometry.  The  difference being that in the former, the basic underlying structure is   $\mathbb{R}$  and in the former it is $\widetilde{\mathbb{R}}_f$. In the complex case $\mathbb{R}$ should be changed by $\mathbb{C}$. The topological ring $\widetilde{\mathbb{R}}_f$ behaves much alike $\mathbb{R}$ in the sense that an element of  $\widetilde{\mathbb{R}}_f$ is either a unit or a zero divisor and its  group of invertible elements, $Inv(\widetilde{\mathbb{R}}_f)$,  is open and dense in $\widetilde{\mathbb{R}}_f$. The structure of $\mathbb{R}$ is well known as is the structure of $\widetilde{\mathbb{R}}_f$. So our approach gives researchers from other field the possibility to apply the theory without having to  dive  into the complicated details of its construction. Basically all one needs is to get acquainted with  $\widetilde{\mathbb{R}}_f$: its ideal structure, idempotents and group of units. We refer the  interested reader  to  \cite{aragona2013algebraic, AJ,   ver1}.

It follows from the embedding theorem above  that certain  problem on $M$ involving distributions and their products or involving certain singularities can be lifted to  problems on $M^*$ involving only  ${\cal{C}}^{\infty}-$functions defined on $M^*$.  For all the environments  involved one can define the support of their elements,  integration, membranes,  internal sets and  obtain topological results. In particular, one can prove that ${\cal{D}}^{\prime}(M)$ is discretely embedded in  $\hat{{\cal{G}}}(M)$ and  that  for   $T\in {\cal{D}}^{\prime}(M)$ and  $\varphi \in{\cal{D}}(M)$  the following equality holds also  in  $\widetilde{\mathbb{R}}_f$ (see previous sections). 

$$\langle T\mid \varphi \rangle =\int_M \kappa(T)\varphi $$

  We now define one more notion of derivation, the notion of  G\^ateau  and Fr\'echet differentiability in the full environments.    Let $\mathbb{K}\in \{\mathbb{R},\mathbb{C}\}$,  $X$ and $Y$ be     topological $\widetilde{\mathbb{K}}_f-$modules and let  $\Omega \subset X$ be an   open subset.  A function $F :\Omega \longrightarrow Y$   is said to be $\widetilde{\mathbb{K}}_f-${\it G\^ateau-differentiable} in $x\in int(\Omega)$, the set of  interior points of $\Omega$,  if  there exists  a   
   $\widetilde{\mathbb{K}}_f-$linear   map $DF (x) : X\longrightarrow Y$ such that  

$$\lim_{t\longrightarrow 0}\frac{F(x+\alpha^\bullet_{-ln(t)}u)-F(x)}{\alpha^\bullet_{-ln(t)}}=DF(x)u, \ \forall \ u\in X$$
 
\noindent Such a differentiable map is said to be  $\widetilde{\mathbb{K}}_f-${\it Fr\'echet-differentiable} in $ x$ if 

$$\lim_{\|u\| \longrightarrow 0}\frac{F(x+u)-F(x)-DF(x)u}{\alpha^\bullet_{-ln(\|u\|)}}=0, \ \forall \ u\in X$$
   	 
	  It is easy to see that the  latter   definition implies the former  and also the
continuity of the map $F$. If the map $DF : X \longrightarrow  {\cal{L}}(X,Y),$ 
 that associates each $x$ with $DF(x)$ is continuous then we say that F is of class ${\cal{C}}^1$ or continuously $\widetilde{\mathbb{K}}_f-$Fr\'{e}chet-differentiable.
 This extends Generalized Differential Calculus to $\widetilde{\mathbb{K}}_f-$modules  making possible  applications in what we  coin {\it Generalized Variational Calculus}.  The  development and applications of these ideas in all environments   will appear elsewhere.    We finish this section looking at events in  Generalized Space-Time and generalized solutions of equations. {\it We do not claim that what follows next corresponds to physical reality. Conclusion are solely based on our interpretation of our model}.

 Let $M$ be     four dimensional Classical Space-Time. Since $M$ is curved, it follows by  the Nash Embedding Theorem and Whitney's Embedding Theorem that   there exists a smallest  $n\in\{5,6,7, 8, 9\}$ such that    $M$ can be isometrically  embedded into   $\mathbb{R}^n$.  Define    Generalized Space-Time to be the ${\cal{G}}_f-$manifold   $M^*$    associated to $M$. We have that $M$ is discretely embedded in $M^*$ and $M\subset M^*\subset \overline{B}_1(0) \subset\widetilde{M}\subset  \widetilde{\mathbb{R}}^n$, with $ssupp(M^*)=M$. The infinitesimals live in $M^*\cap B_1(\vec{0})$ and the infinities live in $\widetilde{M}\cap  (\overline{B}_1(\vec{0}))^c\subset \widetilde{\mathbb{R}}^n$. 
 
 If $p\in M^*$ is an infinitesimal such that  $\|p\|_2\notin Inv(\widetilde{\mathbb{R}})$,  then   there exist $e,f\in {\cal{B}}(\widetilde{\mathbb{R}})$  satisfying  $e\cdot p=\vec{0}$ and $f\cdot\|p\|_2\in Inv(f\cdot\widetilde{\mathbb{R}})$. Given $p_1\in M^*$,  the interleaving $ ep_1+(1-e)p=ep_1+p\in M^*$. So one can interleave two events (points) without, possibly, modifying at least one of them. Once interleaved, they become a single event in $M^*$. In general, an interleaving is of the form  $\sum\limits_j e_j\cdot x_j$, where $x_j\in \widetilde{\mathbb{R}}^n$ and $e_j\in {\cal{B}}(\widetilde{\mathbb{R}})$, $e_i\cdot e_j =\delta_{ij}e_i$ and $\sum\limits_j e_j =1$,  where $\delta_{ij}$ is Kronecker's delta function.  A generalized transition probability $\nu(e_j)$  is associated to each $e_j$  and  $\sum\limits_j \nu(e_j) =1$.    If measuring on $M$ corresponds to applying the function $F$,  then $F( \sum\limits_j e_j\cdot x_j)=\sum\limits_j e_j\cdot F( x_j)$ is again an interleaving.   For example,   let $T=x\delta$,    $x_0, x_j\in\mathbb{R}$, $e_j\in {\cal{B}}(\widetilde{\mathbb{R}})$,  $T_j(x)=T(\frac{x_j}{x_0}x)$   and consider the interleaving  $F=\sum\limits_j T_je_j$. Then $\kappa(T)(x_0\cdot \alpha)=x_0\rho(x_0)$ and $\kappa(F)(x_0\cdot \alpha)=\sum\limits_j \kappa(T_j)(x_0\cdot \alpha)e_j=\sum\limits_j \kappa(T)(x_j\cdot\alpha)e_j=\sum\limits_j x_j\rho(x_j)e_j$ showing  that   an infinitesimal  can produce   a {\it simultaneous  interleaved  effect}  at   points  $x_j$ at arbitrary classical distances  from $\ x_0\alpha\in halo(0)$ (see also \cite{juriaans2022fixed}).        The product of the infinitesimal   $x_0\alpha$  and the infinity $\delta(x_0\alpha)$,   can be  considered  {\it the collapse of the infinity} or  {\it the surge of the infinitesimal}.   Summing the interleaving $x=e_1x_1+e_2x_2$ and $y=f_1y_1+f_2y_2$ results in the interleaving $e_1f_1(x_1+y_1)+e_1f_2(x_1+y_2)+e_2f_1(x_2+y_1)+e_2f_2(x_2+y_2)$, a superposition.   So in $M$ interleaving and their sums are perceived like waves.     Motion of objects in $M^*$ is captured by hypersequences,  which    can be considered to be  displacement with choices,  made along the displacement,  leading to the same point in case of convergence.

  A certain set of idempotents$ -$ which maybe are linked to a certain set of automata $-$ in  $M^*$ seems to  determine  phenomenon $-$ which are the results of interleaving $-$  in  $M$.   That what is not interleaved cannot be observed until it is interleaved.    Since there is a fixed distance between points in  $M$ $-$ it is a grid of equidistant points $-$ it is the time interacted   interleaving between  points of  $M$, that causes the illusion of  distances between classical points to emerge.  Time, being  partially ordered and  having the possibility of taking infinitesimal values, is event dependent. It results in flows in $M^*$.  When objects are observed, this causes   the multiplication of  the observed objects  with an interleaving of infinities and infinitesimals, resulting in new interleaving and possibly modifying the objects.  Transition probabilities, determined by observations and choices, decide how interleaving are formed.  Things in $M$ become meaningful and interact, i.e.,  $M$  expands and attraction occurs, as interleaving are formed. Consequently, in $M^*$  observations and choices are  the driving forces  behind spacetime expansion in $M$ and it is interleaving that results in gravity.     In case $\mathbb{K}=\mathbb{C}$ other interpretations arise, in particular, when this involves infinitesimals.    For a general $M$, managing  to control the inputs in events in $M^*$  and  to direct  the outputs in $M$ might  be  a   tool  to be used  and indicate  a way to  glance into the unseeable, thus  making possible the confirmation of its existence.

   Consider a flow   $F$  in   $M^*$.  In some parts of  $F$  some of its  elements   can behave  as infinitesimals    resulting in some  probability of  their rendezvous with infinities in $\widetilde{M}$.  Taking   a small  open   generalized space-time volume $\Delta$ containing  the rendezvous, there will be some probability that one observes a simultaneous  interleaved effect, caused by the collapse of infinities in $\Delta$,   resulting in  points contained  in $\Delta\cap (F\cap M)$ and, hence,  in   flow patterns of $\Delta \cap F\cap M$ that exhibit chaotic and unpredictable behavior.  The unpredictability  depends  on the    generalized  transition probabilities  of the idempotents of the interleaving involved   and the  type of infinitesimals  in $\Delta$ and infinities in $\widetilde{M}$. Instantaneousness may also occur since it is produced by the observation in $M$ of the collapse of  infinities in $M^*$.
    
   Suppose  that one  interleaves two events $F_1$ and $F_2$ in $M^*$  so that they are contained in   a small  open   generalized space-time volume $\Delta$   stretch in the classical space direction. The creation of infinitesimals at  $F_1\cap M$,  one end of the space-stretch,  may result in  the observation of  an instantaneous  simultaneous  interleaved  effect  at $F_2\cap M$,    the other end of the space-stretch. No contradiction arrises  in the instantaneousness,   since    $M$  consists of a grid of equidistant points  in   $M^*$     all  resulting from   the collapse of   infinities.  In $M^*$, no points of $M$ are far away.

   Other phenomena can be the result of the fact that the support of a generalized event (or generalized solution of an  equation) has more than one point in its support. For certain differential equations this may  lead to different  numerical solutions depending on the numerical refinement or numerical grid being used in the calculations.  Numerical grids  correspond to idempotents and the numerical solution given by   each grid  corresponds  to the product of the generalized solution and the corresponding grid-idempotent.     For such equations, only very fine tuning on grids, which is highly improbable,  will result in the same idempotent.  This suggest that the support of the generalized solution of such an equation   has cluster  points.    This may indicate  a way  to confirm the existence of these generalized environments.   References on which this section is based are \cite{aragona1991intrinsic, aragona2005discontinuous, aragonamembranas, aragona2013algebraic, JJO,  https://doi.org/10.48550/arxiv.1706.02810, garetto2011hilbert, strong, gkos4, invariant, Paulo,  juriaans2022fixed,   ku, ku4,  ros,  oberguggenberger2008internal}.

\section{Conclusion}

The generalized reals were introduced  and   milieus were constructed   resulting in   a notion of Calculus  which naturally extends Newton's and Schwartz's  Calculus,  the main difference being   that classically we stop measuring at scale $\alpha$,  the  natural gauge.    Generalized Differential Geometry is introduced  resulting in  the construction of Generalized Space-Time, which we do not claim corresponds to physical reality.    Below  a certain distance in Generalized Space-Time,  points in Classical Space-Time  cease to exists, the latter becomes fussy,   uncertainty of measurements becomes the rule  and reality becomes sequential, i.e., depending on a history of (infinitely) many interleaved events. As infinities have the same nature, their  rendezvous appears to cause,  seemingly unpredictable,  spooky  interleaved events.  Such  encounters can be considered as the  collapse of the infinities as time  takes infinitesimal values.  For example, the collapse of Dirac's infinity is illustrated by the function $f(t)=t\delta(t)=t_0\rho(t_0)$, where $\delta(t)$  is the speed,  at    $t=  t_0 \alpha\in   \mathbb{R}\cdot \alpha\subset halo(0)$, a grid of equidistant points in Generalized  Time.   The solution $u(t,x)=\frac{x}{t+\alpha}$ of the equation $u_t +uu_x=0, u(0,x)=x\alpha^{-1} $ satisfies  $u(t_0\alpha,x_0\alpha)=\frac{x_0}{1+t_0}\in \mathbb{R}$.  For  $w(t)=     arctan(\alpha^{-1}t) $ we have  $ w^2(t_0\alpha^2)\in halo(0)$,  but  $(w^2)^{\prime}(t_0\alpha^2)\in halo(2t_0)$. Hence,   in infinitesimal time     classical measurements    can be arbitrarily  big.    Consequently, phenomena in Generalized Space-Time  can effect     Classical Space-Time in several  counterintuitive ways, being turbulence, spookiness         and the illusion of distances some of the effects.     The possible dependence on the grid used  to find  numerical solutions   of   certain differential equations is  considered, suggesting an explanation   in terms of the generalized solutions of  these equations, the support of these generalized solutions and grid-idempotents. Finally, the stepping stone of  generalized variational calculus is laid.    Our proposal, centered around the notions of  idempotents, interleaving and support,  is for researchers   to consider  these generalized milieus and concepts  as they might  give  insights concerning questions in  classical environments and provide alternative  mathematical tools.   It might be however that, in terms of computation, Classical Calculus and Generalized Calculus are equivalent, which should explain the enormous  success of Classical Calculus even though infinitesimals and infinities are absent in the classical realm.  However,  more mathematical  effort may be needed and less clarity of  why's might  occur when using  Classical Calculus to describe  physical reality. 
   
\vspace{0.5cm}

\noindent {\bf{Acknowledgements:}} Most parts of this paper stem from  the second author's Ph.D.  thesis done under the guidance of the first author. He is grateful to  IME-USP and the Centro de Ci\^{e}ncias de Balsas-UFMA for their hospitality and the possibility to carry out his Ph.D. as programmed.


\end{document}